\documentclass[11pt]{article}
\usepackage[a4paper,margin=1in]{geometry}
\usepackage[T1]{fontenc}
\usepackage{amsmath,amssymb,amsthm,mathtools,mathrsfs,bm}
\usepackage{bbm}
\usepackage{enumitem}
\usepackage{xcolor}
\usepackage{microtype}
\usepackage[round,authoryear]{natbib}
\usepackage{hyperref}
\usepackage{graphicx}
\usepackage{booktabs}
\usepackage{float}
\usepackage[section]{placeins}
\usepackage{cleveref}
\hypersetup{hidelinks}
\allowdisplaybreaks
\newtheorem{theorem}{Theorem}
\newtheorem{proposition}[theorem]{Proposition}
\newtheorem{lemma}[theorem]{Lemma}
\newtheorem{corollary}[theorem]{Corollary}
\theoremstyle{definition}

\newtheorem{remark}[theorem]{Remark}

\newcommand{\R}{\mathbb{R}}
\newcommand{\E}{\mathbb{E}}
\newcommand{\Var}{\mathrm{Var}}

\newcommand{\tr}{\mathrm{tr}}
\newcommand{\op}{\mathrm{op}}
\newcommand{\F}{\mathrm{F}}
\newcommand{\dd}{\mathrm{d}}
\newcommand{\dto}{\xrightarrow[]{d}}
\newcommand{\norm}[1]{\left\lVert #1 \right\rVert}
\newcommand{\abs}[1]{\left| #1 \right|}
\newcommand{\Ln}{L_n}

\newcommand{\Toe}[1]{\mathcal{T}_n(#1)}
\newcommand{\Cmn}{C_n^{\mathrm m}}
\newcommand{\Dmn}{D_n^{\mathrm m}}
\newcommand{\Con}{C_n^{\mathrm o}}
\newcommand{\Don}{D_n^{\mathrm o}}
\newcommand{\Aon}{A_n^{\mathrm o}}
\newcommand{\eul}{\gamma_{\mathrm E}}
\newcommand{\digammaf}{\psi}

\title{\textsc{Boundary Inference for Mixed Fractional Models under High-Frequency Observation}\\
Critical LAN and Score Tests at $H=3/4$}
\author{
Cai Chunhao\\
Department of Mathematics, Sun Yat Sen University\\
\texttt{caichh9@mail.sysu.edu.cn}
\and
Shang Yiwu\\
School of Mathematical Sciences, Nankai University\\
\texttt{shanyw@mail.nankai.edu.cn}
\and
Xiao Weilin\\
Department of Mathematics, Zhejiang University\\
\texttt{wlxiao@zju.edu.cn}
\and
Zhang Cong\\
Department of Mathematics, Zhejiang University\\
\texttt{zcong@zju.edu.cn}
}
\date{}

\begin{document}
\maketitle

\begin{abstract}
 We study boundary inference at $H=3/4$ for mixed fractional Brownian motion and mixed fractional Ornstein--Uhlenbeck models under high-frequency observation. This boundary is economically important because it separates the critical and supercritical regimes of mixed fractional dynamics.

 We make three contributions. First, we identify the exact critical first-order scaling and show that, after removing the explicit linear component in the $H$-score, the transformed $(\sigma,H)$ block is already non-degenerate. Second, we establish critical score central limit theorems (CLT) and derive local asymptotic normality (LAN) with fully explicit leading information constants for both models. Third, we construct boundary-calibrated one-sided score tests for detecting entry into the supercritical region $H>3/4$ and discuss feasible implementation through restricted nuisance estimation. Monte Carlo evidence shows that the feasible statistic has the correct directional power but conservative null calibration. Finally, an intraday illustration on one-minute SPY data finds no persistent evidence in favor of $H>3/4$.
\end{abstract}



\section{Introduction}
Mixed fractional models sit naturally at the intersection of continuous-time probability and financial econometrics. On the one hand, they enrich Brownian dynamics by adding a fractional component capable of carrying persistent dependence; on the other hand, they remain close enough to Gaussian likelihood theory to support high-frequency inference. In empirical work, such models are appealing whenever one wants to describe serially dependent microstructure distortions, persistent volatility signals, or reduced-form deviations from the benchmark semimartingale paradigm. For mixed fractional Brownian motion, however, the semimartingale interpretation is known to hinge on the critical threshold $H=3/4$ in \citet{Cheridito2001}. The point $H=3/4$ is therefore not merely a mathematical edge case. It is the place where the probabilistic status of the model, and hence its econometric interpretation, changes in a decisive way.

That question has become more rather than less relevant in recent years. The modern rough-volatility literature has made fractional behavior a central empirical theme in finance, with influential contributions in \emph{Quantitative Finance} and related outlets emphasizing rough volatility surfaces, short-maturity asymptotics, and the high-frequency evidence for very small Hurst exponents, see, e.g., \citet{BayerFrizGatheral2016}, \citet{Fukasawa2017}, \citet{GatheralJaissonRosenbaum2018}, and \citet{Fukasawa2021}. At the same time, the \emph{Journal of Econometrics} has developed a complementary line of work that treats fractional and rough dynamics as econometric objects to be estimated, tested, and used in forecasting. Examples include the fOU realized-volatility framework of \citet{WangXiaoYu2023}, the GMM-based roughness estimation strategy of \citet{BolkoEtAl2023}, and the approximate Whittle likelihood analysis of \citet{ShiYuZhang2024}. Much of this literature focuses, understandably, on rough regimes with $H<1/2$ or on estimation away from critical boundaries. Yet if mixed fractional models are to be used as diagnostic or structural tools, then the opposite side of the problem matters as well: one needs a statistically disciplined way to decide whether the data are still compatible with the critical or subcritical side of the model, or whether they have crossed into the supercritical region where the semimartingale logic becomes substantially more delicate.

This is precisely where the present paper enters. Existing LAN results for high-frequency fractional models do not resolve the boundary point $H=3/4$. For fractional Gaussian noise, \citet{BrousteFukasawa2018} established LAN under high-frequency observation. For mixed fractional Brownian motion and mixed fractional Ornstein--Uhlenbeck processes, \citet{Cai2026} and \citet{Cai2026mfou} showed that the local geometry differs sharply across the subcritical and supercritical regions. But the critical case itself is excluded from those analyses. At $H=3/4$, both off-boundary normalizations fail. The leading traces acquire an additional logarithmic factor, the raw $H$-score develops an explicit singular coupling with the nuisance direction, and neither side of the existing theory can simply be pasted onto the boundary.

Our main observation is that the critical point is singular only before the right coordinates are chosen. Once the explicit linear term in the $H$-score is removed, the remaining $(\sigma,H)$ block is already non-degenerate at $H=3/4$. This fact turns out to be enough to recover an exact critical normalization, to write down the limiting covariance structure with explicit constants, and to obtain a single triangular local reparametrization for both mixed fractional experiments. In that sense, the paper does not merely fill in a missing boundary case. It shows that the boundary has its own coherent local asymptotic geometry, and that this geometry leads directly to a usable score-based inferential problem.

The testing question that emerges is economically and econometrically transparent. We study
\[
H_0:\ H\le \frac34
\qquad\text{versus}\qquad
H_1:\ H>\frac34,
\]
so the null collects the critical and subcritical configurations while the alternative corresponds to entry into the supercritical regime. When mixed fractional models are interpreted as reduced-form descriptions of persistent high-frequency dependence, rejection in favor of $H>3/4$ can be read as evidence of an exceptionally persistent regime, one that is harder to reconcile with the semimartingale benchmark underlying much of continuous-time financial econometrics. Conversely, non-rejection supports staying on the critical side of the boundary, where the local asymptotic structure is fundamentally different. This is the sense in which our score tests should be understood: they are not only formal consequences of LAN expansions, but boundary diagnostics for a class of models that sits directly beside the rough-volatility and high-frequency econometrics literatures.

The contribution of the paper is therefore theoretical in method but econometric in orientation. We identify the exact critical first-order scales, prove Gaussian limits for transformed score vectors in both the mfBm and mfOU experiments, derive explicit LAN expansions with fully specified leading constants, and use the resulting central sequences to construct one-sided score tests calibrated at the boundary. We also discuss feasible implementation through nuisance estimation under the restriction $H=3/4$, report an expanded Monte Carlo analysis of the feasible mfBm statistic, and close with a compact intraday empirical illustration built around that same boundary score.

The rest of the paper is organized as follows. Section~\ref{sec:preliminaries} introduces the mfBm and mfOU models, derives the critical low-frequency constants and covariance matrices at $H=3/4$, and establishes the exact score decompositions that drive the subsequent analysis. Section~\ref{sec:main-results} states the transformed score CLTs, the critical LAN expansions. Sections~\ref{sec:proofofcltscore} and~\ref{sec:proofofLANTHM} contain the technical proofs. Section~\ref{sec:score-tests} develops the boundary score tests and explains feasible implementation under nuisance estimation. Section~\ref{sec:simulation} reports the expanded Monte Carlo evidence. Section~\ref{sec:empirical} turns to empirical implementation of the exact feasible score on genuinely intraday data. The appendices collect the low-frequency reductions, explicit constants and higher-order remainder bounds.

\section{Model setup, critical objects and score decomposition}\label{sec:preliminaries}
We now introduce the objects that will carry the boundary analysis. Rather than introducing the models and the asymptotic ingredients in isolation, it is helpful to keep the inferential goal in view from the outset. Everything that follows is organized around one question: what remains of the Gaussian score structure exactly at $H=3/4$, once the usual off-boundary normalizations have ceased to be informative? For that reason, the section moves from the model definitions to the low-frequency behavior that drives the critical rates, and then from those rates to the score decomposition that eventually makes the boundary experiment tractable.
\subsection{The mfBm model}
Let
\[
Y_t=\sigma B_t^H+W_t,
\qquad t\ge0,
\]
where $B_t^H$ is a fractional Brownian motion, $W_t$ is an independent standard Brownian motion. The parameters $(\sigma,H)$ are assumed known. We observe the increment vector
\[
X_n^{\mathrm m}:=(\Delta_1^nY,\dots,\Delta_n^nY)^\top,
\qquad
\Delta_k^nY:=Y_{k\Delta_n}-Y_{(k-1)\Delta_n},
\]
where the superscript $\mathrm m$ denotes that the observations are from the mfBm model. Its covariance matrix is
\[
\Sigma_n^{\mathrm m}(\sigma,H)
=\Delta_n\Bigl(I_n+\sigma^2\Delta_n^{2H-1}\Toe{f_H}\Bigr),
\]
where $f_H$ is the spectral density of fractional Gaussian noise. Under the Fourier convention fixed in notation, one may take
\begin{equation}\label{eq:fgn-sd-exact}
f_H(\lambda)
=
2\,\Gamma(2H+1)\sin(\pi H)\,(1-\cos\lambda)\sum_{k\in\mathbb Z}\abs{\lambda+2\pi k}^{-(2H+1)},
\qquad \lambda\in[-\pi,\pi]\setminus\{0\}.
\end{equation}
\begin{proposition}
\label{prop:Km-from-sd}
Let
\[
c_H:=\Gamma(2H+1)\sin(\pi H).
\]
Then, as $\lambda\to0$,
\[
f_H(\lambda)=c_H\,|\lambda|^{1-2H}\bigl(1+o(1)\bigr)
\]
and
\[
\partial_H f_H(\lambda)
=
f_H(\lambda)\Bigl(2\psi(2H+1)+\pi\cot(\pi H)-2\log|\lambda|+o(1)\Bigr),
\]
where $\psi$ denotes the digamma function.
In particular, at $H=3/4$,
\[
K_{\mathrm m}:=c_{3/4},
\qquad
\beta_{\mathrm m}:=2\psi\!\left(\frac52\right)-\pi,
\]
and $\beta_{\mathrm m}$ is the derivative of $\log c_H$ at $H=3/4$.
\end{proposition}
\begin{proof}
Starting from \eqref{eq:fgn-sd-exact}, the term $k=0$ dominates as $\lambda\to0$, while the terms $k\neq0$ are $O(1)$. Since
\[
1-\cos\lambda=\frac{\lambda^2}{2}+O(\lambda^4),
\]
we obtain
\[
f_H(\lambda)
=
2c_H\left(\frac{\lambda^2}{2}+O(\lambda^4)\right)|\lambda|^{-(2H+1)}\bigl(1+o(1)\bigr)
=
c_H|\lambda|^{1-2H}\bigl(1+o(1)\bigr).
\]
Differentiating exact series representation \eqref{eq:fgn-sd-exact} with respect to $H$ gives
\[
\partial_H\log f_H(\lambda)
=
\frac{\dd}{\dd H}\log c_H-2\log|\lambda|+o(1)
=
2\digammaf(2H+1)+\pi\cot(\pi H)-2\log|\lambda|+o(1),
\]
which proves the proposition.
\end{proof}
 Proposition \ref{prop:Km-from-sd} identifies the low frequency behavior of the fGn spectral density and its derivative in the Hurst direction.  In particular, the constants $K_m$ and $\beta_m$ extracted here will enter the critical score profiles, the covariance matrix, and the subsequent LAN analysis.

At the critical point $H=3/4$ the standard fGn low frequency coefficient is
\begin{equation}\label{eq:Km}
K_{\mathrm m}
:=\Gamma\!\left(\frac52\right)\sin\!\left(\frac{3\pi}{4}\right)
=\frac{3\sqrt{2\pi}}{8},
\end{equation}
\begin{equation}\label{eq:betam}
    \beta_{\mathrm m}
:=\left.\frac{\dd}{\dd H}\log\!\bigl(\Gamma(2H+1)\sin(\pi H)\bigr)\right|_{H=3/4}
=2\digammaf\!\left(\frac52\right)-\pi
=\frac{16}{3}-2\eul-4\log 2-\pi.
\end{equation}
Thus,
\begin{equation}\label{eq:mfbm-lowfreq}
 f_{3/4}(\lambda)=K_{\mathrm m}\abs{\lambda}^{-1/2}\bigl(1+r_{\mathrm m}(\lambda)\bigr),
 \qquad
 \partial_H f_H\big|_{H=3/4}(\lambda)
 =f_{3/4}(\lambda)\bigl(\beta_{\mathrm m}-2\log\abs{\lambda}+\rho_{\mathrm m}(\lambda)\bigr),
\end{equation}
with $r_{\mathrm m}(\lambda)\to0$ and $\rho_{\mathrm m}(\lambda)\to0$ as $\lambda\to0$. \\
Fix a critical parameter point
\[
\vartheta_{\mathrm m}:=(\sigma,3/4),\qquad \sigma>0.
\]
All mfBm score matrices below are evaluated at $\vartheta_{\mathrm m}$.

\subsection{The mfOU model}
Let $X$ be the stationary solution of
\[
\dd X_t=-\alpha X_t\,\dd t+\dd M_t^H,
\qquad
M_t^H=\sigma B_t^H+W_t,
\qquad \alpha>0.
\]
We observe
\[
X_n^{\mathrm o}:=(X_{\Delta_n},\dots,X_{n\Delta_n})^\top.
\]
Let $f_{\Delta}^{\mathrm o}(\lambda;\sigma,H,\alpha)$ be the sampled spectral density of the stationary mfOU model (see \cite{Cai2026mfou}). At the continuous time frequencies, the stationary mfOU solution has a smooth OU transfer factor multiplying the same fractional singular coefficient $c_H=\Gamma(2H+1)\sin(\pi H)$ as in \eqref{eq:fgn-sd-exact}. Consequently, the sampled spectral density inherits the same low frequency singular coefficient and logarithmic derivative in the $H$-direction. More precisely, under the normalization used throughout the paper,
\[
f_{\Delta}^{\mathrm o}(\lambda;\sigma,H,\alpha)
=
c_H\,|\lambda|^{1-2H}\bigl(1+o(1)\bigr)
\qquad (\lambda\to0),
\]
and
\[
\partial_H f_{\Delta}^{\mathrm o}(\lambda;\sigma,H,\alpha)
=
f_{\Delta}^{\mathrm o}(\lambda;\sigma,H,\alpha)
\Bigl(2\digammaf(2H+1)+\pi\cot(\pi H)-2\log|\lambda|+o(1)\Bigr).
\]
Therefore the relevant critical low frequency coefficients are
\begin{equation}\label{eq:kappa-o}
K_{\mathrm o}:=\Gamma\!\left(\frac52\right)\sin\!\left(\frac{3\pi}{4}\right)=\frac{3\sqrt{2\pi}}{8}=K_{\mathrm m},
\end{equation}
and
\begin{equation}\label{eq:beta-o}
\beta_{\mathrm o}:=\left.\frac{\dd}{\dd H}\log\!\bigl(\Gamma(2H+1)\sin(\pi H)\bigr)\right|_{H=3/4}
=2\digammaf\!\left(\frac52\right)-\pi
=\frac{16}{3}-2\eul-4\log 2-\pi
=\beta_{\mathrm m}.
\end{equation}
Only the induced critical score profiles are used below, so no additional second-order expansion of $f_{\Delta}^{\mathrm o}$ is required in this paper.

Fix $\vartheta_{\mathrm o}:=(\sigma,3/4,\alpha), \sigma>0, \alpha>0$, all mfOU score matrices below are evaluated at $\vartheta_{\mathrm o}$.

\subsection{Critical covariance constants}
 This subsection collects the critical covariance constants that determine the covariance matrices in the score CLTs at $H=3/4$. These quantities also enter the information matrices appearing in the subsequent LAN expansions and score tests. Define
\begin{align}
\Xi_{\sigma\sigma}^{\mathrm m}&:=\frac{4\sigma^2K_{\mathrm m}^2}{\pi}=\frac{9\sigma^2}{8}, &
\Xi_{\sigma H}^{\mathrm m}&:=\frac{2\sigma^3K_{\mathrm m}^2}{\pi}=\frac{9\sigma^3}{16}, &
\Xi_{HH}^{\mathrm m}&:=\frac{4\sigma^4K_{\mathrm m}^2}{3\pi}=\frac{3\sigma^4}{8}, \label{eq:xi-m-main}\\
\Xi_{\sigma\sigma}^{\mathrm o}&:=\frac{4\sigma^2K_{\mathrm o}^2}{\pi}=\frac{9\sigma^2}{8}, &
\Xi_{\sigma H}^{\mathrm o}&:=\frac{2\sigma^3K_{\mathrm o}^2}{\pi}=\frac{9\sigma^3}{16}, &
\Xi_{HH}^{\mathrm o}&:=\frac{4\sigma^4K_{\mathrm o}^2}{3\pi}=\frac{3\sigma^4}{8}, &
\Xi_{\alpha\alpha}^{\mathrm o}&:=\frac{1}{\alpha}. \label{eq:xi-o-main}
\end{align}
Set
\[
\Gamma_{ab}^{\mathrm m}:=\frac12\Xi_{ab}^{\mathrm m},
\qquad
\Gamma_{ab}^{\mathrm o}:=\frac12\Xi_{ab}^{\mathrm o}.
\]
The critical covariance matrices of mfBm and mfOU are
\begin{equation}\label{eq:gamma-m-main}
\Gamma_{\mathrm m}^{\mathrm{crit}}
:=
\begin{pmatrix}
\Gamma_{\sigma\sigma}^{\mathrm m} & \Gamma_{\sigma H}^{\mathrm m}\\
\Gamma_{\sigma H}^{\mathrm m} & \Gamma_{HH}^{\mathrm m}
\end{pmatrix}
=
\begin{pmatrix}
\dfrac{9\sigma^2}{16} & \dfrac{9\sigma^3}{32}\\[1ex]
\dfrac{9\sigma^3}{32} & \dfrac{3\sigma^4}{16}
\end{pmatrix},
\end{equation}
and
\begin{equation}\label{eq:gamma-o-main}
\Gamma_{\mathrm o}^{\mathrm{crit}}
:=
\begin{pmatrix}
\Gamma_{\sigma\sigma}^{\mathrm o} & \Gamma_{\sigma H}^{\mathrm o} & 0\\
\Gamma_{\sigma H}^{\mathrm o} & \Gamma_{HH}^{\mathrm o} & 0\\
0 & 0 & \Gamma_{\alpha\alpha}^{\mathrm o}
\end{pmatrix}
=
\begin{pmatrix}
\dfrac{9\sigma^2}{16} & \dfrac{9\sigma^3}{32} & 0\\[1ex]
\dfrac{9\sigma^3}{32} & \dfrac{3\sigma^4}{16} & 0\\[1ex]
0 & 0 & \dfrac{1}{2\alpha}
\end{pmatrix}.
\end{equation}
In both models,
\begin{equation}\label{eq:critical-corr}
\frac{\Gamma_{\sigma H}^{\bullet}}{\sqrt{\Gamma_{\sigma\sigma}^{\bullet}\Gamma_{HH}^{\bullet}}}
=\frac{\sqrt3}{2}<1,
\qquad \bullet\in\{\mathrm m,\mathrm o\}.
\end{equation}
Hence the critical $(\sigma,H)$-block is already non-degenerate after the first transformation. Its connection with the transformed score will be made explicit in the next subsection. The low frequency reductions and the explicit one-dimensional integrals leading to \eqref{eq:xi-m-main}--\eqref{eq:xi-o-main} are collected in Appendix~\ref{app:profiles}.

\begin{corollary}\label{cor:nondeg}
The matrix $\Gamma_{\mathrm m}^{\mathrm{crit}}$ is positive definite, and so is the $2\times2$ upper-left block of $\Gamma_{\mathrm o}^{\mathrm{crit}}$.
\end{corollary}

\begin{proof}
The correlation identity \eqref{eq:critical-corr} gives
\[
\det\Gamma_{\mathrm m}^{\mathrm{crit}}>0,
\qquad
\det\!\left(\Gamma_{\mathrm o}^{\mathrm{crit}}\big|_{\{\sigma,H\}}\right)>0.
\]
\end{proof}

\begin{remark}
The proposition \ref{prop:Km-from-sd} together with the reductions in Appendix A shows that the critical coefficients come directly from the near zero behavior of the spectral density. $K_{\mathrm m}$ and $\beta_{\mathrm m}$ are obtained by evaluating the coefficient $c_H=\Gamma(2H+1)\sin(\pi H)$ and its logarithmic derivative at $H=3/4$. The corresponding constants $K_{\mathrm o}$ and $\beta_{\mathrm o}$ are obtained analogously for the sampled mfOU density. With these closed forms, every leading LAN constant  (see theorems \ref{thm:mfbm-lan} and \ref{thm:mfou-lan}) is explicit
\[
K_{\mathrm m}=K_{\mathrm o}=\frac{3\sqrt{2\pi}}{8},
\qquad
\beta_{\mathrm m}=\beta_{\mathrm o}=\frac{16}{3}-2\eul-4\log 2-\pi,
\]
\[
\Gamma_{\mathrm m}^{\mathrm{crit}}\big|_{\{\sigma,H\}}
=
\Gamma_{\mathrm o}^{\mathrm{crit}}\big|_{\{\sigma,H\}}
=
\begin{pmatrix}
\dfrac{9\sigma^2}{16} & \dfrac{9\sigma^3}{32}\\[1ex]
\dfrac{9\sigma^3}{32} & \dfrac{3\sigma^4}{16}
\end{pmatrix},
\qquad
I_H^{\mathrm{eff,m}}=I_H^{\mathrm{eff,o}}=\frac{3\sigma^4}{64}.
\]
where \[
I_H^{\mathrm{eff},\bullet}
=
\Gamma_{HH}^\bullet
-
\Gamma_{H \sigma }^\bullet
\bigl(\Gamma_{\sigma \sigma }^\bullet\bigr)^{-1}
\Gamma_{\sigma H}^\bullet,\qquad \bullet\in\{\mathrm m,\mathrm o\}.\] Thus the critical LAN expansion and the score-test normalization do not contain any unnamed leading constants.
\end{remark}
\subsection{Exact score functions}
This subsection establishes the exact score representations at the critical point. These formulas will be used in the next section to derive the trace asymptotics, score CLTs and LAN expansions.
\begin{proposition} \label{prop:exact-scores}
Let $\Sigma_n(\theta)$ be the relevant covariance matrix and define
\[
Z_n:=\Sigma_n(\theta)^{-1/2}X_n\sim \mathcal{N}(0,I_n).
\]
Then every Gaussian score function $S_{i,n}(\theta)$ has the exact representation
\begin{equation}\label{eq:exact-score}
S_{i,n}=\frac12Q_n(M_{i,n}),
\qquad
M_{i,n}:=\Sigma_n(\theta)^{-1/2}\,\partial_{\theta_i}\Sigma_n(\theta)\,\Sigma_n(\theta)^{-1/2}
\end{equation}
where $Q_n(A):=Z_n^{\top}AZ_n-\tr(A)$ denotes the centered Gaussian quadratic form.
\end{proposition}

\begin{proof}
The Gaussian score identity gives
\[
S_{i,n}(\theta)=\partial_{\theta_i}\ell_n(\theta)
=-\frac12\tr\!\bigl(\Sigma_n(\theta)^{-1}\partial_{\theta_i}\Sigma_n(\theta)\bigr)
+\frac12X_n^\top \Sigma_n(\theta)^{-1}(\partial_{\theta_i}\Sigma_n(\theta))\Sigma_n(\theta)^{-1}X_n.
\]
Using $\E[X_n^\top AX_n]=\tr(A\Sigma_n(\theta))$ yields \eqref{eq:exact-score}.
\end{proof}We now rewrite the raw score functions at the critical boundary \(H=3/4\) into exact decompositions that isolate the singular \(\sigma\)-direction in the \(H\)-score. We refer to this step as the first transformation. This yields the transformed \(H\)-scores. For mfBm define
\begin{equation} \label{eq:Cm-def}
    \Cmn:=M_{\sigma,n}^{\mathrm m},
\qquad
M_{H,n}^{\mathrm m}:=\Sigma_n^{\mathrm m}(\theta)^{-1/2}\,\partial_H\Sigma_n^{\mathrm m}(\theta)\,\Sigma_n^{\mathrm m}(\theta)^{-1/2},
\end{equation}
all evaluated at $\theta=\vartheta_{\mathrm m}$. Since
\[
\partial_H\bigl(\sigma^2\Delta_n^{2H-1}\bigr)=2\sigma^2\Delta_n^{2H-1}\log\Delta_n=-2\sigma^2\Delta_n^{2H-1}\Ln,
\]
one gets the exact splitting
\begin{equation}\label{eq:mfbm-split}
S_{H,n}^{\mathrm m}=-\sigma\Ln\,S_{\sigma,n}^{\mathrm m}+R_{H,n}^{\mathrm m},
\qquad
R_{H,n}^{\mathrm m}=\frac12Q_n(\Dmn),
\qquad
\Dmn:=M_{H,n}^{\mathrm m}+\sigma\Ln\,\Cmn.
\end{equation}
The positive sign in the definition of $\Dmn$ comes from replacing $\log\Delta_n$ by $-\Ln$, this sign convention keeps the critical decomposition unambiguous.\\
Equation~\eqref{eq:mfbm-split} also defines what we call the first transformation at the critical boundary. We replace the raw pair $\bigl(S_{\sigma,n}^{\mathrm m},S_{H,n}^{\mathrm m}\bigr)$ by the transformed pair $\bigl(S_{\sigma,n}^{\mathrm m},R_{H,n}^{\mathrm m}\bigr)$, thereby removing the explicit singular coupling term $-\sigma\Ln\,S_{\sigma,n}^{\mathrm m}$ from the $H$-score.\\
For mfOU define
\[
\Con:=M_{\sigma,n}^{\mathrm o},
\qquad
\Don:=M_{H,n}^{\mathrm o}+\sigma\Ln\,M_{\sigma,n}^{\mathrm o},
\qquad
\Aon:=M_{\alpha,n}^{\mathrm o},
\]
all evaluated at $\theta=\vartheta_{\mathrm o}$. Then
\begin{equation}\label{eq:mfou-split}
S_{H,n}^{\mathrm o}=-\sigma\Ln\,S_{\sigma,n}^{\mathrm o}+R_{H,n}^{\mathrm o},
\qquad
R_{H,n}^{\mathrm o}=\frac12Q_n(\Don),
\qquad
S_{\alpha,n}^{\mathrm o}=\frac12Q_n(\Aon).
\end{equation}
Accordingly, we refer to $(S^m_{\sigma,n},R^m_{H,n})^\top$ in the mfBm case and $(S^o_{\sigma,n},R^o_{H,n},S^o_{\alpha,n})^\top$ in the mfOU case as the transformed score vectors.

Here, we define the term critical profile as the leading reduced frequency-domain symbol of the corresponding score at the critical point $H=3/4$.
\paragraph{Critical profiles.}
Let
\begin{equation}\label{eq:weight-m}
w_{\mathrm m}(u):=\frac{\eta_{\mathrm m}\abs{u}^{-1/2}}{1+\eta_{\mathrm m}\abs{u}^{-1/2}},
\qquad
\eta_{\mathrm m}:=\sigma^2K_{\mathrm m},
\end{equation}
and
\begin{equation}\label{eq:weight-o}
w_{\mathrm o}(u):=\frac{\eta_{\mathrm o}\abs{u}^{-1/2}}{1+\eta_{\mathrm o}\abs{u}^{-1/2}},
\qquad
\eta_{\mathrm o}:=\sigma^2K_{\mathrm o}.
\end{equation}
The critical score profiles are
\begin{equation}\label{eq:sigma-profiles}
g_{\sigma}^{\mathrm m}(u):=\frac{2}{\sigma}w_{\mathrm m}(u),
\qquad
g_{\sigma}^{\mathrm o}(u):=\frac{2}{\sigma}w_{\mathrm o}(u),
\end{equation}
and
\begin{equation}\label{eq:h-profiles}
g_{H,n}^{\mathrm m}(u):=w_{\mathrm m}(u)\bigl(2\Ln+\beta_{\mathrm m}-2\log\abs{u}\bigr),
\qquad
g_{H,n}^{\mathrm o}(u):=w_{\mathrm o}(u)\bigl(2\Ln+\beta_{\mathrm o}-2\log\abs{u}\bigr).
\end{equation}
For the mfOU drift parameter,
\begin{equation}\label{eq:alpha-profile}
g_{\alpha}^{\mathrm o}(u):=-\frac{2\alpha}{\alpha^2+u^2}.
\end{equation}
The additional factor $2\Ln$ in \eqref{eq:h-profiles} survives after removing the explicit $-\sigma\Ln\,S_{\sigma,n}$ term. It comes from the derivative of the low frequency singularity itself.

\begin{lemma} \label{lem:tail-orders}
As $\abs{u}\to\infty$,
\[
w_{\mathrm m}(u)\sim \eta_{\mathrm m}\abs{u}^{-1/2},
\qquad
w_{\mathrm o}(u)\sim \eta_{\mathrm o}\abs{u}^{-1/2}.
\]
Consequently,
\begin{align*}
\bigl(g_{\sigma}^{\mathrm m}(u)\bigr)^2 &\sim \frac{4\eta_{\mathrm m}^2}{\sigma^2}\abs{u}^{-1},
&
\bigl(g_{\sigma}^{\mathrm o}(u)\bigr)^2 &\sim \frac{4\eta_{\mathrm o}^2}{\sigma^2}\abs{u}^{-1},\\
\bigl(g_{H,n}^{\mathrm m}(u)\bigr)^2 &\sim \eta_{\mathrm m}^2\abs{u}^{-1}\bigl(2\Ln+\beta_{\mathrm m}-2\log\abs{u}\bigr)^2,
&
\bigl(g_{H,n}^{\mathrm o}(u)\bigr)^2 &\sim \eta_{\mathrm o}^2\abs{u}^{-1}\bigl(2\Ln+\beta_{\mathrm o}-2\log\abs{u}\bigr)^2,
\end{align*}
and
\begin{align*}
g_{\sigma}^{\mathrm m}(u)g_{H,n}^{\mathrm m}(u)
&\sim \frac{2\eta_{\mathrm m}^2}{\sigma}\abs{u}^{-1}\bigl(2\Ln+\beta_{\mathrm m}-2\log\abs{u}\bigr),\\
g_{\sigma}^{\mathrm o}(u)g_{H,n}^{\mathrm o}(u)
&\sim \frac{2\eta_{\mathrm o}^2}{\sigma}\abs{u}^{-1}\bigl(2\Ln+\beta_{\mathrm o}-2\log\abs{u}\bigr),\\
g_{\alpha}^{\mathrm o}(u) &\sim -2\alpha\abs{u}^{-2}.
\end{align*}
\end{lemma}

\begin{proof}
This is immediate from \eqref{eq:weight-m}--\eqref{eq:alpha-profile}.
\end{proof}
In short, the critical profiles \(g_{\sigma}^{\bullet}(u)\), \(g_{H,n}^{\bullet}(u)\) (with \(\bullet\in\{\mathrm{m},\mathrm{o}\}\)) and \(g_{\alpha}^{\mathrm{o}}(u)\) capture the low-frequency behavior of the score matrices at the boundary \(H=3/4\). They translate matrix traces into explicit integrals over the scaled frequency \(u\). Consequently, all leading constants in the score CLTs and LAN information matrices are expressed in closed form via integrals of these profiles, providing a unified description of the boundary experiment.
\section{Main asymptotic results} \label{sec:main-results}
\subsection{Critical score central limit theorems}

With the score decomposition in hand, we can now pass from exact identities to their asymptotic meaning at the boundary $H=3/4$. Throughout this section,
\[
L_n:=\log(1/\Delta_n),\qquad \Delta_n\downarrow0,\qquad n\Delta_n\to\infty,
\]
and all score matrices and score functions are evaluated at the critical parameter point. For the mfBm experiment we write
\[
\vartheta_m:=(\sigma,3/4),\qquad \sigma>0,
\]
while for the mfOU experiment we write
\[
\vartheta_o:=(\sigma,3/4,\alpha),\qquad \sigma>0,\ \alpha>0.
\]
The decisive feature proved in Section~\ref{sec:preliminaries} is that the raw $H$-score does not need to be discarded at the boundary; it needs to be recentered in the right direction. More precisely,
\[
S^m_{H,n}=-\sigma L_n S^m_{\sigma,n}+R^m_{H,n},
\qquad
S^o_{H,n}=-\sigma L_n S^o_{\sigma,n}+R^o_{H,n}.
\]
The explicit $\sigma L_n$ term is the only first-order obstruction. Once it is peeled away, the residual score has its own critical fluctuation scale and can be handled jointly with the nuisance directions. The limit theorems below show that this transformed system has a non-degenerate Gaussian limit for both mixed fractional experiments. In other words, the boundary is singular in the original coordinates but regular enough in the transformed ones to support a full LAN analysis and, later on, a score test.

\begin{theorem}
  \label{thm:mfbm-main}
Consider the mixed fractional Brownian motion model at the critical point
\(
\vartheta_m=(\sigma,3/4)
\).
Then, as \(n\to\infty\),
\[
\left(
\frac{S_{\sigma,n}^{\mathrm m}}{\sqrt{n\Delta_n\Ln}},
\frac{R_{H,n}^{\mathrm m}}{\sqrt{n\Delta_n}\,\Ln^{3/2}}
\right)
\dto\mathcal{N}\bigl(0,\Gamma_{\mathrm m}^{\mathrm{crit}}\bigr)
\]
where $\Gamma_{\mathrm m}^{\mathrm{crit}}$ is defined by equation \eqref{eq:gamma-m-main}.
In particular, the critical fluctuation scale of the \(\sigma\)-score is
\(
(n\Delta_nL_n)^{1/2}
\),
whereas the critical fluctuation scale of the transformed \(H\)-score is
\(
(n\Delta_nL_n^3)^{1/2}
\).
\end{theorem}

\begin{theorem}
\label{thm:mfou-main}
Consider the mixed fractional Ornstein--Uhlenbeck model at the critical point
\(
\vartheta_o=(\sigma,3/4,\alpha)
\).
Then, as \(n\to\infty\),
\[
\left(
\frac{S_{\sigma,n}^{\mathrm o}}{\sqrt{n\Delta_n\Ln}},
\frac{R_{H,n}^{\mathrm o}}{\sqrt{n\Delta_n}\,\Ln^{3/2}},
\frac{S_{\alpha,n}^{\mathrm o}}{\sqrt{n\Delta_n}}
\right)
\dto\mathcal{N}\bigl(0,\Gamma_{\mathrm o}^{\mathrm{crit}}\bigr).
\]
where $\Gamma_{\mathrm o}^{\mathrm{crit}}$ is defined by equation \eqref{eq:gamma-o-main}. In particular, the \((\sigma,H)\)-block and the \(\alpha\)-direction separate at the leading order, so the drift parameter does not interfere with the first-order boundary geometry.
\end{theorem}
\subsection{Critical LAN expansions}
Theorems \ref{thm:mfbm-main} and \ref{thm:mfou-main} are deliberately written in terms of the transformed $H$-score rather than the raw one. This is not a cosmetic choice. At the boundary, the transformed coordinate system is the one in which the limit experiment becomes visible without further ad hoc projections. The following proposition records that coordinate change in matrix form and makes clear why a single triangular transformation is enough.
\begin{proposition}\label{prop:main-transform}
Let
\[
L_n^{\mathrm m}:=
\begin{pmatrix}
1 & 0\\
\sigma\Ln & 1
\end{pmatrix},
\qquad
L_n^{\mathrm o}:=
\begin{pmatrix}
1 & 0 & 0\\
\sigma\Ln & 1 & 0\\
0 & 0 & 1
\end{pmatrix}.
\]
Then
\[
L_n^{\mathrm m}
\begin{pmatrix}
S_{\sigma,n}^{\mathrm m}\\
S_{H,n}^{\mathrm m}
\end{pmatrix}
=
\begin{pmatrix}
S_{\sigma,n}^{\mathrm m}\\
R_{H,n}^{\mathrm m}
\end{pmatrix},
\qquad
L_n^{\mathrm o}
\begin{pmatrix}
S_{\sigma,n}^{\mathrm o}\\
S_{H,n}^{\mathrm o}\\
S_{\alpha,n}^{\mathrm o}
\end{pmatrix}
=
\begin{pmatrix}
S_{\sigma,n}^{\mathrm o}\\
R_{H,n}^{\mathrm o}\\
S_{\alpha,n}^{\mathrm o}
\end{pmatrix}.
\]
Thus, at the critical boundary \(H=3/4\), the explicit singular coupling in the \(\sigma\)-direction is removed by a single triangular transformation, and no second projection is required.
\end{proposition}
\begin{proof}
This is a direct restatement of \eqref{eq:mfbm-split} and \eqref{eq:mfou-split} in matrix form. Moreover, by \eqref{eq:critical-corr},
\[
\frac{\Gamma^\bullet_{\sigma H}}{\sqrt{\Gamma^\bullet_{\sigma\sigma}\Gamma^\bullet_{HH}}}
=\frac{\sqrt3}{2}<1,
\qquad \bullet\in\{\mathrm m,\mathrm o\},
\]
so the critical \((\sigma,H)\)-block is already non-degenerate after the first transformation. Therefore no second projection is required.
\end{proof}
For later use in the LAN expansion, we introduce the corresponding critical rate matrices
\[
r^{-1}_{n,m}
:=
\begin{pmatrix}
(n\Delta_nL_n)^{-1/2} & \sigma (n\Delta_nL_n)^{-1/2}\\
0 & (n\Delta_n)^{-1/2}L_n^{-3/2}
\end{pmatrix},
\]
and
\[
r^{-1}_{n,o}
:=
\begin{pmatrix}
(n\Delta_nL_n)^{-1/2} & \sigma (n\Delta_nL_n)^{-1/2} & 0\\
0 & (n\Delta_n)^{-1/2}L_n^{-3/2} & 0\\
0 & 0 & (n\Delta_n)^{-1/2}
\end{pmatrix}.
\]
These matrices encode exactly the critical local scales of the parameters together with the first-order coupling between the \(\sigma\)- and \(H\)-directions. In this form, the local perturbation already reflects the geometry revealed by the score decomposition, and the log-likelihood admits the standard quadratic LAN representation.
\begin{theorem} \label{thm:mfbm-lan}
Let $\ell_n^{\mathrm m}$ be the Gaussian log-likelihood of the mfBm increment experiment and define
\[
\Xi_n^{\mathrm m}:=
\begin{pmatrix}
\dfrac{S_{\sigma,n}^{\mathrm m}}{\sqrt{n\Delta_n\Ln}}\\[0.8ex]
\dfrac{R_{H,n}^{\mathrm m}}{\sqrt{n\Delta_n}\,\Ln^{3/2}}
\end{pmatrix},
\qquad
\vartheta_{\mathrm m,n}(h):=\vartheta_{\mathrm m}+r_{n,\mathrm m}^{-1}h,
\qquad h\in\R^2.
\]
Then for every fixed $h\in\R^2$,
\[
\ell_n^{\mathrm m}(\vartheta_{\mathrm m,n}(h))-
\ell_n^{\mathrm m}(\vartheta_{\mathrm m})
=
h^\top\Xi_n^{\mathrm m}
-\frac12 h^\top\Gamma_{\mathrm m}^{\mathrm{crit}}h
+o_{P_{\vartheta_{\mathrm m}}}(1).
\]
In particular, the mfBm experiment is LAN at $\vartheta_{\mathrm m}$ with asymptotic information $\Gamma_{\mathrm m}^{\mathrm{crit}}$.
\end{theorem}

\begin{theorem} \label{thm:mfou-lan}
Let $\ell_n^{\mathrm o}$ be the Gaussian log-likelihood of the stationary mfOU experiment and define
\[
\Xi_n^{\mathrm o}:=
\begin{pmatrix}
\dfrac{S_{\sigma,n}^{\mathrm o}}{\sqrt{n\Delta_n\Ln}}\\[0.8ex]
\dfrac{R_{H,n}^{\mathrm o}}{\sqrt{n\Delta_n}\,\Ln^{3/2}}\\[0.8ex]
\dfrac{S_{\alpha,n}^{\mathrm o}}{\sqrt{n\Delta_n}}
\end{pmatrix},
\qquad
\vartheta_{\mathrm o,n}(h):=\vartheta_{\mathrm o}+r_{n,\mathrm o}^{-1}h,
\qquad h\in\R^3.
\]
Then for every fixed $h\in\R^3$,
\[
\ell_n^{\mathrm o}(\vartheta_{\mathrm o,n}(h))-
\ell_n^{\mathrm o}(\vartheta_{\mathrm o})
=
h^\top\Xi_n^{\mathrm o}
-\frac12 h^\top\Gamma_{\mathrm o}^{\mathrm{crit}}h
+o_{P_{\vartheta_{\mathrm o}}}(1).
\]
In particular, the mfOU experiment is LAN at $\vartheta_{\mathrm o}$ with asymptotic information $\Gamma_{\mathrm o}^{\mathrm{crit}}$.
\end{theorem}

\section{Technical proof of the score CLTs} \label{sec:proofofcltscore}
This section establishes the score central limit theorems announced above. The argument follows the logic already suggested by the critical decomposition: the Toeplitz structure first converts the relevant quadratic forms into explicit trace asymptotics, and those asymptotics are then paired with a Gaussian quadratic-form CLT to obtain the limiting normal laws.
\subsection{Critical trace asymptotics}\label{sec:trace-asymptotics}

\begin{proposition}\label{prop:mfbm-traces}
Consider the mixed fractional Brownian motion model at the critical point
\[
\vartheta_{\mathrm m}=(\sigma,3/4),\qquad \sigma>0,
\]
under the high-frequency asymptotic regime
\[
\Delta_n\downarrow 0,\qquad n\Delta_n\to\infty,\qquad L_n:=\log(1/\Delta_n).
\]
Then the following trace asymptotics hold:
\begin{align}
\tr\bigl((\Cmn)^2\bigr) &\sim \Xi_{\sigma\sigma}^{\mathrm m}\,n\Delta_n\Ln, \label{eq:mfbm-c2}\\
\tr(\Cmn\Dmn) &\sim \Xi_{\sigma H}^{\mathrm m}\,n\Delta_n\Ln^2, \label{eq:mfbm-cd}\\
\tr\bigl((\Dmn)^2\bigr) &\sim \Xi_{HH}^{\mathrm m}\,n\Delta_n\Ln^3, \label{eq:mfbm-d2}
\end{align}
and
\begin{equation}\label{eq:mfbm-op}
\norm{\Cmn}_{\op}=O(1),
\qquad
\norm{\Dmn}_{\op}=O(\log n).
\end{equation}
Here and throughout, the notation is as in Section~2.
\end{proposition}

\begin{proof}
At \(H=3/4\), after the low frequency rescaling \(\lambda=\Delta_n u\), the critical profiles corresponding to \(\Cmn\) and \(\Dmn\) are \(g_\sigma^{\mathrm m}(u)\) and \(g_{H,n}^{\mathrm m}(u)\). Therefore,
\[
\tr\bigl((\Cmn)^2\bigr),\qquad \tr(\Cmn\Dmn),\qquad \tr\bigl((\Dmn)^2\bigr)
\]
are reduced to the corresponding one-dimensional critical integrals via the standard Toeplitz trace reduction (see \cite{Cai2026}). \\
For \eqref{eq:mfbm-c2}, split the integral into $\abs{u}\le1$ and $1\le\abs{u}\le 1/\Delta_n$. The bounded part is $O(1)$. By lemma \ref{lem:tail-orders},
\[
\int_1^{1/\Delta_n}\bigl(g_{\sigma}^{\mathrm m}(u)\bigr)^2\,\dd u
\sim \frac{4\eta_{\mathrm m}^2}{\sigma^2}\int_1^{1/\Delta_n}u^{-1}\,\dd u
=\frac{4\eta_{\mathrm m}^2}{\sigma^2}\Ln.
\]
Doubling by evenness and multiplying by $n\Delta_n/(2\pi)$ gives
\[
\int_{-1/\Delta_n}^{1/\Delta_n}\bigl(g_{\sigma}^{\mathrm m}(u)\bigr)^2\,\dd u
\sim \frac{8\eta_{\mathrm m}^2}{\sigma^2}\Ln,
\]
hence \eqref{eq:mfbm-c2}.\\
For \eqref{eq:mfbm-cd}, the bounded region contributes only $O(\Ln)$ and is therefore negligible after division by $\Ln^2$. On $[1,1/\Delta_n]$,
\[
\int_1^{1/\Delta_n}g_{\sigma}^{\mathrm m}(u)g_{H,n}^{\mathrm m}(u)\,\dd u
\sim \frac{2\eta_{\mathrm m}^2}{\sigma}
\int_1^{1/\Delta_n}u^{-1}\bigl(2\Ln+\beta_{\mathrm m}-2\log u\bigr)\,\dd u.
\]
With $t=\log u$,
\[
\int_1^{1/\Delta_n}u^{-1}\bigl(2\Ln+\beta_{\mathrm m}-2\log u\bigr)\,\dd u
=\int_0^{\Ln}(2\Ln+\beta_{\mathrm m}-2t)\,\dd t
=\Ln^2+\beta_{\mathrm m}\Ln.
\]
After doubling by evenness, \eqref{eq:mfbm-cd} follows.\\
For \eqref{eq:mfbm-d2},
\[
\int_1^{1/\Delta_n}\bigl(g_{H,n}^{\mathrm m}(u)\bigr)^2\,\dd u
\sim \eta_{\mathrm m}^2\int_1^{1/\Delta_n}u^{-1}\bigl(2\Ln+\beta_{\mathrm m}-2\log u\bigr)^2\,\dd u.
\]
Again $t=\log u$ gives
\begin{align*}
\int_1^{1/\Delta_n}u^{-1}\bigl(2\Ln+\beta_{\mathrm m}-2\log u\bigr)^2\,\dd u
&=\int_0^{\Ln}(2\Ln+\beta_{\mathrm m}-2t)^2\,\dd t\\
&=\frac{4}{3}\Ln^3+2\beta_{\mathrm m}\Ln^2+\beta_{\mathrm m}^2\Ln.
\end{align*}
Hence the leading even integral has coefficient $8\eta_{\mathrm m}^2/3$, which yields \eqref{eq:mfbm-d2}.

For $C_n^m$, the operator bound follows from the boundedness of the corresponding reduced symbol. For $D_n^m$, however, the reduced profile $g_{H,n}$ is not bounded near the origin because of the logarithmic term. Hence the bound $||D_n^m||_{op}=O(\log n)$ requires a separate argument at the level of the original sandwiched symbol, which we omit here and the proof follows the similar argument as in \cite{Cai2026,Cai2026mfou}.
\end{proof}
\begin{proposition}\label{prop:mfou-traces}
Consider the mixed fractional Ornstein--Uhlenbeck model at the critical point
\[
\vartheta_{\mathrm o}=(\sigma,3/4,\alpha),\qquad \sigma>0,\ \alpha>0,
\]
under the high-frequency asymptotic regime
\[
\Delta_n\downarrow 0,\qquad n\Delta_n\to\infty,\qquad L_n:=\log(1/\Delta_n).
\]
Then the following trace asymptotics hold:
\begin{align}
\tr\bigl((\Con)^2\bigr) &\sim \Xi_{\sigma\sigma}^{\mathrm o}\,n\Delta_n\Ln, \label{eq:mfou-c2}\\
\tr(\Con\Don) &\sim \Xi_{\sigma H}^{\mathrm o}\,n\Delta_n\Ln^2, \label{eq:mfou-cd}\\
\tr\bigl((\Don)^2\bigr) &\sim \Xi_{HH}^{\mathrm o}\,n\Delta_n\Ln^3, \label{eq:mfou-d2}\\
\tr\bigl((\Aon)^2\bigr) &\sim \Xi_{\alpha\alpha}^{\mathrm o}\,n\Delta_n, \label{eq:mfou-a2}
\end{align}
and
\begin{equation}\label{eq:mfou-op}
\norm{\Con}_{\op}=O(1),\qquad
 \norm{\Don}_{\op}=O(\log n),\qquad
\norm{\Aon}_{\op}=O(1).
\end{equation}
Moreover,
\begin{equation}\label{eq:mfou-alpha-mixed}
\tr(\Con\Aon)=O(n\Delta_n),
\qquad
\tr(\Don\Aon)=O(n\Delta_n\Ln).
\end{equation}
Here and throughout, the notation is as in Section~2.
\end{proposition}

\begin{proof}
For the $(\sigma,H)$-block, the critical rescaling of the sampled mfOU spectral density yields the same integrals as in the mfBm case (see \cite{Cai2026mfou}), with $\eta_{\mathrm m}$ replaced by $\eta_{\mathrm o}$ (see Appendix~\ref{app:profiles} for the explicit low frequency reduction). This proves \eqref{eq:mfou-c2}--\eqref{eq:mfou-d2}.\\
For the drift score,
\[
\int_{\R}\bigl(g_{\alpha}^{\mathrm o}(u)\bigr)^2\,\dd u
=4\alpha^2\int_{\R}\frac{\dd u}{(\alpha^2+u^2)^2}
=\frac{2\pi}{\alpha},
\]
which gives \eqref{eq:mfou-a2} after multiplication by $n\Delta_n/(2\pi)$.\\
For the mixed drift traces, $g_{\sigma}^{\mathrm o}$ is bounded near the origin and behaves like $\abs{u}^{-1/2}$ at infinity, whereas $g_{\alpha}^{\mathrm o}(u)$ behaves like $\abs{u}^{-2}$. Hence
\[
\int_{\R}\abs{g_{\sigma}^{\mathrm o}(u)g_{\alpha}^{\mathrm o}(u)}\,\dd u<\infty,
\]
so $\tr(\Con\Aon)=O(n\Delta_n)$. Likewise,
\[
\int_{\R}\abs{g_{H,n}^{\mathrm o}(u)g_{\alpha}^{\mathrm o}(u)}\,\dd u
\le C(1+\Ln)\int_{\R}\frac{\dd u}{1+u^2}<\infty,
\]
which yields $\tr(\Don\Aon)=O(n\Delta_n\Ln)$. The operator bounds in \eqref{eq:mfou-op} follow from the same profile bounds as in the mfBm case.
\end{proof}

\subsection{CLT for the transformed score vector}\label{sec:critical-clts}
The proof is based on the operator/Frobenius criterion for Gaussian quadratic forms.

\begin{lemma} \label{lem:qf-clt}
Let $A_n$ be real symmetric matrices and let $Z_n\sim \mathcal{N}(0,I_n)$. If
\[
\tr(A_n^2)\to\tau^2\in(0,\infty),
\qquad
\frac{\norm{A_n}_{\op}}{\norm{A_n}_{\F}}\to0,
\]
then
\[
\frac12Q_n(A_n)\dto \mathcal{N}\!\left(0,\frac{\tau^2}{2}\right).
\]
\end{lemma}

\begin{proof}
 See \cite{Cai2026}.
\end{proof}

\begin{proof}[\textbf{Proof of theorem \ref{thm:mfbm-main}}]
Fix $u=(u_1,u_2)\in\R^2$ and define
\[
A_n^{\mathrm m}(u):=
\frac{u_1}{\sqrt{n\Delta_n\Ln}}\Cmn
+
\frac{u_2}{\sqrt{n\Delta_n}\,\Ln^{3/2}}\Dmn
\]
where $C_n^m$ and $D_n^m$ are defined in \eqref{eq:Cm-def} and \eqref{eq:mfbm-split}. Then
\[
\frac{u_1S_{\sigma,n}^{\mathrm m}}{\sqrt{n\Delta_n\Ln}}
+
\frac{u_2R_{H,n}^{\mathrm m}}{\sqrt{n\Delta_n}\,\Ln^{3/2}}
=\frac12Q_n\bigl(A_n^{\mathrm m}(u)\bigr).
\]
By proposition \ref{prop:mfbm-traces},
\[
\tr\Bigl(\bigl(A_n^{\mathrm m}(u)\bigr)^2\Bigr)
\to u^\top\bigl(2\Gamma_{\mathrm m}^{\mathrm{crit}}\bigr)u.
\]
Moreover,
\[
\|A_n^{\mathrm m}(u)\|_{\op}
\le
\frac{|u_1|}{\sqrt{n\Delta_nL_n}}\|C_n^{\mathrm m}\|_{\op}
+
\frac{|u_2|}{\sqrt{n\Delta_n}L_n^{3/2}}\|D_n^{\mathrm m}\|_{\op}
=
O\!\left(\frac1{\sqrt{n\Delta_nL_n}}\right)
+
O\!\left(\frac{\log n}{\sqrt{n\Delta_n}\,L_n^{3/2}}\right).
\]
Since
\[
\log n = L_n + \log(n\Delta_n),
\]
we have
\[
\frac{\log n}{\sqrt{n\Delta_n}\,L_n^{3/2}}
=
O\!\left(\frac1{\sqrt{n\Delta_nL_n}}\right)
+
\frac{\log(n\Delta_n)}{\sqrt{n\Delta_n}\,L_n^{3/2}}
=o(1),
\]
because \(n\Delta_n\to\infty\) and \(L_n\to\infty\). Therefore
\[
\|A_n^{\mathrm m}(u)\|_{\op}\to0.
\]
In addition,
\[
\|A_n^{\mathrm m}(u)\|_{\mathrm F}^2
\to
2u^\top \Gamma_{\mathrm m}^{\mathrm{crit}}u.
\] Since $\Gamma_{\mathrm m}^{\mathrm{crit}}$ is positive definite, the operator/Frobenius ratio tends to $0$ for every $u\neq0$. By lemma \ref{lem:qf-clt},
\[
\frac12Q_n\bigl(A_n^{\mathrm m}(u)\bigr)
\dto\mathcal{N}\bigl(0,u^\top\Gamma_{\mathrm m}^{\mathrm{crit}}u\bigr).
\]
The Cram\'er--Wold device proves the theorem.
\end{proof}

\begin{proof}[\textbf{Proof of theorem \ref{thm:mfou-main}}]
Fix $u=(u_1,u_2,u_3)\in\R^3$ and define
\[
A_n^{\mathrm o}(u):=
\frac{u_1}{\sqrt{n\Delta_n\Ln}}\Con
+
\frac{u_2}{\sqrt{n\Delta_n}\,\Ln^{3/2}}\Don
+
\frac{u_3}{\sqrt{n\Delta_n}}\Aon.
\]
Then
\[
\frac{u_1S_{\sigma,n}^{\mathrm o}}{\sqrt{n\Delta_n\Ln}}
+
\frac{u_2R_{H,n}^{\mathrm o}}{\sqrt{n\Delta_n}\,\Ln^{3/2}}
+
\frac{u_3S_{\alpha,n}^{\mathrm o}}{\sqrt{n\Delta_n}}
=\frac12Q_n\bigl(A_n^{\mathrm o}(u)\bigr).
\]
By proposition \ref{prop:mfou-traces}, the diagonal terms contribute $2u^\top\Gamma_{\mathrm o}^{\mathrm{crit}}u$ at the trace level. The mixed traces with $\Aon$ vanish after normalization because
\[
\frac{n\Delta_n}{\sqrt{n\Delta_n\Ln}\sqrt{n\Delta_n}}=\Ln^{-1/2}\to0,
\qquad
\frac{n\Delta_n\Ln}{\sqrt{n\Delta_n}\,\Ln^{3/2}\sqrt{n\Delta_n}}=\Ln^{-1/2}\to0.
\]
Hence
\[
\tr\Bigl(\bigl(A_n^{\mathrm o}(u)\bigr)^2\Bigr)
\to 2u^\top\Gamma_{\mathrm o}^{\mathrm{crit}}u.
\]
Also,
\[
\|A_n^{\mathrm o}(u)\|_{\op}
\le
\frac{|u_1|}{\sqrt{n\Delta_nL_n}}\|C_n^{\mathrm o}\|_{\op}
+
\frac{|u_2|}{\sqrt{n\Delta_n}L_n^{3/2}}\|D_n^{\mathrm o}\|_{\op}
+
\frac{|u_3|}{\sqrt{n\Delta_n}}\|A_n^{\mathrm o}\|_{\op}
\]
\[
=
O\!\left(\frac1{\sqrt{n\Delta_nL_n}}\right)
+
O\!\left(\frac{\log n}{\sqrt{n\Delta_n}\,L_n^{3/2}}\right)
+
O\!\left(\frac1{\sqrt{n\Delta_n}}\right)\xrightarrow[]{n\rightarrow\infty}0.
\].

Since $\Gamma_{\mathrm o}^{\mathrm{crit}}$ is positive definite, the Frobenius norm of $A_n^{\mathrm o}(u)$ stays away from $0$ for every $u\neq0$, and therefore
\[
\frac{\norm{A_n^{\mathrm o}(u)}_{\op}}{\norm{A_n^{\mathrm o}(u)}_{\F}}\to0.
\]
Lemma \ref{lem:qf-clt} and the Cram\'er--Wold device conclude the proof.
\end{proof}
\section{Technical proof of the LAN theorems}\label{sec:proofofLANTHM}
In this section we combine the CLTs for the transformed score vectors with a second-order expansion of the Gaussian likelihood ratio. We first state the matrix Taylor bounds. We then linearize the local covariance perturbation. Finally, we identify the linear score term, the quadratic information term and the remainder terms. At the critical point \(H=3/4\), a single triangular rate matrix already yields a non-degenerate quadratic term.
\subsection{Matrix Taylor expansions}\label{subsc:matrix-taylor}
Subsection \ref{subsc:matrix-taylor} collects the matrix expansions used in the LAN proof.

\begin{lemma} \label{lem:matrix-taylor}
Let $S$ be a real symmetric matrix with $\norm{S}_{\op}\le 1/2$. Then
\[
\log\det(I+S)=\tr(S)-\frac12\tr(S^2)+R_{\log}(S),
\qquad
\abs{R_{\log}(S)}\le C\norm{S}_{\op}\tr(S^2),
\]
and
\[
(I+S)^{-1}=I-S+S^2+R_{\mathrm{inv}}(S),
\qquad
R_{\mathrm{inv}}(S)=\sum_{k\ge3}(-1)^kS^k,
\]
with
\[
\norm{R_{\mathrm{inv}}(S)}_{\op}\le C\norm{S}_{\op}^3,
\qquad
\tr\bigl(R_{\mathrm{inv}}(S)^2\bigr)\le C\tr(S^6).
\]
Here $C>0$ is a universal constant.
\end{lemma}

\begin{proof}
Diagonalize $S$. Since every eigenvalue belongs to $[-1/2,1/2]$, the scalar Taylor expansions of $\log(1+x)$ and $(1+x)^{-1}$ can be summed over the spectrum term by term (see \cite{Cai2026mfou}).
\end{proof}

\subsection{Critical local perturbation}
This subsection linearizes the local covariance perturbation under the critical local parameterization at \(H=3/4\). For the local parameter shift \(\vartheta_n(h)\), we write the sandwiched perturbation \(S_n(h)\)
as its first-order part \(B_n(h)\) plus a negligible remainder. This is the key  for the
likelihood expansion in the next subsection.
\begin{proposition} \label{prop:critical-linearization}
Fix $h\in\R^2$ and set
\[
\vartheta_{\mathrm m,n}(h)=\vartheta_{\mathrm m}+r_{n,\mathrm m}^{-1}h,
\qquad
B_n^{\mathrm m}(h):=
\frac{h_1}{\sqrt{n\Delta_n\Ln}}\Cmn
+
\frac{h_2}{\sqrt{n\Delta_n}\,\Ln^{3/2}}\Dmn.
\]
Define the sandwiched covariance perturbation
\[
\mathcal S_n^{\mathrm m}(h):=
\Sigma_n^{\mathrm m}(\vartheta_{\mathrm m})^{-1/2}
\bigl(\Sigma_n^{\mathrm m}(\vartheta_{\mathrm m,n}(h))-\Sigma_n^{\mathrm m}(\vartheta_{\mathrm m})\bigr)
\Sigma_n^{\mathrm m}(\vartheta_{\mathrm m})^{-1/2}.
\]
Then
\[
\mathcal S_n^{\mathrm m}(h)=B_n^{\mathrm m}(h)+\mathcal R_n^{\mathrm m}(h),
\qquad
\norm{\mathcal S_n^{\mathrm m}(h)}_{\op}\to0,
\qquad
\tr\Bigl(\bigl(\mathcal R_n^{\mathrm m}(h)\bigr)^2\Bigr)\to0.
\]

Fix $h\in\R^3$ and set
\[
\vartheta_{\mathrm o,n}(h)=\vartheta_{\mathrm o}+r_{n,\mathrm o}^{-1}h,
\qquad
B_n^{\mathrm o}(h):=
\frac{h_1}{\sqrt{n\Delta_n\Ln}}\Con
+
\frac{h_2}{\sqrt{n\Delta_n}\,\Ln^{3/2}}\Don
+
\frac{h_3}{\sqrt{n\Delta_n}}\Aon.
\]
Define $\mathcal S_n^{\mathrm o}(h)$ analogously. Then
\[
\mathcal S_n^{\mathrm o}(h)=B_n^{\mathrm o}(h)+\mathcal R_n^{\mathrm o}(h),
\qquad
\norm{\mathcal S_n^{\mathrm o}(h)}_{\op}\to0,
\qquad
\tr\Bigl(\bigl(\mathcal R_n^{\mathrm o}(h)\bigr)^2\Bigr)\to0.
\]
\end{proposition}

\begin{proof}
We start with mfBm. Write $\delta_n^{\mathrm m}(h)=r_{n,\mathrm m}^{-1}h$, so that
\[
\delta_{\sigma,n}^{\mathrm m}=O\bigl((n\Delta_n\Ln)^{-1/2}\bigr),
\qquad
\delta_{H,n}^{\mathrm m}=O\bigl((n\Delta_n)^{-1/2}\Ln^{-3/2}\bigr).
\]
A second-order Taylor expansion around $\vartheta_{\mathrm m}$ yields
\[
\Sigma_n^{\mathrm m}(\vartheta_{\mathrm m,n}(h))-
\Sigma_n^{\mathrm m}(\vartheta_{\mathrm m})
=
\delta_{\sigma,n}^{\mathrm m}\,\partial_\sigma\Sigma_n^{\mathrm m}
+
\delta_{H,n}^{\mathrm m}\,\partial_H\Sigma_n^{\mathrm m}
+
\frac12\sum_{i,j\in\{\sigma,H\}}\delta_{i,n}^{\mathrm m}\delta_{j,n}^{\mathrm m}
\partial_{ij}^2\Sigma_n^{\mathrm m}(\widetilde\vartheta_n),
\]
for some $\widetilde\vartheta_n$ on the segment joining $\vartheta_{\mathrm m}$ and $\vartheta_{\mathrm m,n}(h)$. Sandwiching by $\Sigma_n^{\mathrm m}(\vartheta_{\mathrm m})^{-1/2}$ and using the exact decomposition \eqref{eq:mfbm-split} gives
\[
\mathcal S_n^{\mathrm m}(h)=B_n^{\mathrm m}(h)+\mathcal R_n^{\mathrm m}(h),
\qquad
\mathcal R_n^{\mathrm m}(h)=\frac12\sum_{i,j\in\{\sigma,H\}}\delta_{i,n}^{\mathrm m}\delta_{j,n}^{\mathrm m}
M_{ij,n}^{\mathrm m}(\widetilde\vartheta_n),
\]
where
\[
M_{ij,n}^{\mathrm m}(\theta):=
\Sigma_n^{\mathrm m}(\vartheta_{\mathrm m})^{-1/2}
\partial_{ij}^2\Sigma_n^{\mathrm m}(\theta)
\Sigma_n^{\mathrm m}(\vartheta_{\mathrm m})^{-1/2}.
\]
Differentiating
\[
\Sigma_n^{\mathrm m}(\sigma,H)=\Delta_n\Bigl(I_n+\sigma^2\Delta_n^{2H-1}\Toe{f_H}\Bigr)
\]
term by term, and using the higher-order derivative bounds proved later in proposition \ref{prop:higher-order-derivatives}, we obtain uniformly on shrinking neighborhoods of $\vartheta_{\mathrm m}$
\begin{align*}
\norm{M_{\sigma\sigma,n}^{\mathrm m}}_{\op}&=O(1), &
\norm{M_{\sigma H,n}^{\mathrm m}}_{\op}&=O(\log n), &
\norm{M_{HH,n}^{\mathrm m}}_{\op}&=O((\log n)^2), \\
\tr\bigl((M_{\sigma\sigma,n}^{\mathrm m})^2\bigr)&=O(n\Delta_n\Ln), &
\tr\bigl((M_{\sigma H,n}^{\mathrm m})^2\bigr)&=O(n\Delta_n\Ln^3), &
\tr\bigl((M_{HH,n}^{\mathrm m})^2\bigr)&=O(n\Delta_n\Ln^5).
\end{align*}
Equivalently, each second-order Taylor contribution is already below the LAN scale:
\begin{align*}
(\delta_{\sigma,n}^{\mathrm m})^2\|M_{\sigma\sigma,n}^{\mathrm m}\|_{\op}
&=O\!\left((n\Delta_nL_n)^{-1}\right),\\
|\delta_{\sigma,n}^{\mathrm m}\delta_{H,n}^{\mathrm m}|\|M_{\sigma H,n}^{\mathrm m}\|_{\op}
&=O\!\left(\frac{\log n}{n\Delta_n\,L_n^{2}}\right),\\
(\delta_{H,n}^{\mathrm m})^2\|M_{HH,n}^{\mathrm m}\|_{\op}
&=O\!\left(\frac{(\log n)^2}{n\Delta_n\,L_n^{3}}\right).
\end{align*}

and, at the Frobenius level,
\begin{align*}
(\delta_{\sigma,n}^{\mathrm m})^4\tr\bigl((M_{\sigma\sigma,n}^{\mathrm m})^2\bigr)
&=O\bigl((n\Delta_n\Ln)^{-1}\bigr), \\
(\delta_{\sigma,n}^{\mathrm m}\delta_{H,n}^{\mathrm m})^2\tr\bigl((M_{\sigma H,n}^{\mathrm m})^2\bigr)
&=O\bigl((n\Delta_n\Ln)^{-1}\bigr), \\
(\delta_{H,n}^{\mathrm m})^4\tr\bigl((M_{HH,n}^{\mathrm m})^2\bigr)
&=O\bigl((n\Delta_n\Ln)^{-1}\bigr).
\end{align*}
The $HH$-term is generated by the three critical profiles
\[
\Delta_n^{3/2}\Ln^2\Toe{f_{3/4}},
\qquad
\Delta_n^{3/2}\Ln\Toe{\partial_Hf_H|_{H=3/4}},
\qquad
\Delta_n^{3/2}\Toe{\partial_{HH}^2f_H|_{H=3/4}},
\]
and the last one has low frequency order $|\lambda|^{-1/2}(1+|\log|\lambda||^2)$, so it falls under exactly the same critical trace mechanism as the first-order score matrices. Therefore
\[
\norm{\mathcal R_n^{\mathrm m}(h)}_{\op}=O\Bigl(\frac{(\log n)^2}{n\Delta_n\,L_n^{3}} \Bigr),
\qquad
\tr\Bigl(\bigl(\mathcal R_n^{\mathrm m}(h)\bigr)^2\Bigr)=O\bigl((n\Delta_n\Ln)^{-1}\bigr)\to0.
\]
Since equation \eqref{eq:mfbm-op} implies $\norm{B_n^{\mathrm m}(h)}_{\op}\to0$, we also obtain $\norm{\mathcal S_n^{\mathrm m}(h)}_{\op}\to0$.

For mfOU we write $\delta_n^{\mathrm o}(h)=r_{n,\mathrm o}^{-1}h$ and repeat the same argument. The local Taylor expansion gives
\[
\mathcal S_n^{\mathrm o}(h)=B_n^{\mathrm o}(h)+\mathcal R_n^{\mathrm o}(h),
\qquad
\mathcal R_n^{\mathrm o}(h)=\frac12\sum_{i,j\in\{\sigma,H,\alpha\}}\delta_{i,n}^{\mathrm o}\delta_{j,n}^{\mathrm o}
M_{ij,n}^{\mathrm o}(\widetilde\vartheta_n).
\]
Because the sampled mfOU symbol has the same critical $|u|^{-1/2}$ singularity as in the mfBm case, each $H$-derivative contributes one logarithmic factor, whereas $\alpha$-derivatives only act on the smooth OU correction and therefore preserve the integrability class. Thus
\begin{align*}
\norm{M_{\sigma\sigma,n}^{\mathrm o}}_{\op}&=O(1), &
\norm{M_{\sigma H,n}^{\mathrm o}}_{\op}&=O(\log n), &
\norm{M_{HH,n}^{\mathrm o}}_{\op}&=O((\log n)^2), \\
\norm{M_{\sigma\alpha,n}^{\mathrm o}}_{\op}&=O(1), &
\norm{M_{H\alpha,n}^{\mathrm o}}_{\op}&=O(\log n), &
\norm{M_{\alpha\alpha,n}^{\mathrm o}}_{\op}&=O(1),
\end{align*}
while
\begin{align*}
\tr\bigl((M_{\sigma\sigma,n}^{\mathrm o})^2\bigr)&=O(n\Delta_n\Ln), &
\tr\bigl((M_{\sigma H,n}^{\mathrm o})^2\bigr)&=O(n\Delta_n\Ln^3), &
\tr\bigl((M_{HH,n}^{\mathrm o})^2\bigr)&=O(n\Delta_n\Ln^5), \\
\tr\bigl((M_{\sigma\alpha,n}^{\mathrm o})^2\bigr)&=O(n\Delta_n), &
\tr\bigl((M_{H\alpha,n}^{\mathrm o})^2\bigr)&=O(n\Delta_n\Ln^2), &
\tr\bigl((M_{\alpha\alpha,n}^{\mathrm o})^2\bigr)&=O(n\Delta_n).
\end{align*}
Hence the pure $(\sigma,H)$ contributions are treated exactly as in mfBm, whereas the additional mixed terms satisfy
\begin{align*}
(\delta_{\sigma,n}^{\mathrm o}\delta_{\alpha,n}^{\mathrm o})\norm{M_{\sigma\alpha,n}^{\mathrm o}}_{\op}
&=O\Bigl(\frac{1}{n\Delta_n L^{1/2}_n}\Bigr), \\
(\delta_{H,n}^{\mathrm o}\delta_{\alpha,n}^{\mathrm o})\norm{M_{H\alpha,n}^{\mathrm o}}_{\op}
&=O\Bigl(\frac{\log n }{n\Delta_n L^{3/2}_n}\Bigr), \\
(\delta_{\alpha,n}^{\mathrm o})^2\norm{M_{\alpha\alpha,n}^{\mathrm o}}_{\op}
&=O\Bigl(\frac{1}{n\Delta_n  }\Bigr),
\end{align*}
with the same orders for the corresponding squared Frobenius contributions. Consequently
\[
\norm{\mathcal R_n^{\mathrm o}(h)}_{\op}=O\!\left(
\frac{
1+ {(\log n)^2}/{L_n^3}
}{n\Delta_n}
\right),
\qquad
\tr\Bigl(\bigl(\mathcal R_n^{\mathrm o}(h)\bigr)^2\Bigr)=O\bigl((n\Delta_n)^{-1}\bigr)\to0.
\]

Since equation \eqref{eq:mfou-op} gives $\norm{B_n^{\mathrm o}(h)}_{\op}\to0$, the mfOU claim follows as well.

\end{proof}

\subsection{Final proof of the LAN theorems}

\begin{proof}[\textbf{Proof of theorem \ref{thm:mfbm-lan}}]
Fix $h\in\R^2$ and write
\[
\Sigma_n:=\Sigma_n^{\mathrm m}(\vartheta_{\mathrm m}),
\qquad
\Sigma_{n,h}:=\Sigma_n^{\mathrm m}(\vartheta_{\mathrm m,n}(h)),
\qquad
Z_n:=\Sigma_n^{-1/2}X_n^{\mathrm m}\sim \mathcal{N}(0,I_n).
\]
Define
\[
\mathcal S_n(h):=\Sigma_n^{-1/2}(\Sigma_{n,h}-\Sigma_n)\Sigma_n^{-1/2}.
\]
The standard Gaussian likelihood identity gives
\begin{equation}\label{eq:mfbm-likelihood-identity}
\ell_n^{\mathrm m}(\vartheta_{\mathrm m,n}(h))-
\ell_n^{\mathrm m}(\vartheta_{\mathrm m})
=
-\frac12\log\det\bigl(I_n+\mathcal S_n(h)\bigr)
-\frac12 Z_n^\top\Bigl((I_n+\mathcal S_n(h))^{-1}-I_n\Bigr)Z_n.
\end{equation}
By proposition \ref{prop:critical-linearization} and corollary \ref{cor:cubic-remainder}, the sandwiched local covariance admits a second-order approximation with an additional cubic remainder that is already negligible, and in particular $\norm{\mathcal S_n(h)}_{\op}\to0$. Hence for $n$ large enough we may apply lemma \ref{lem:matrix-taylor} to obtain
\begin{align}
\ell_n^{\mathrm m}(\vartheta_{\mathrm m,n}(h))-
\ell_n^{\mathrm m}(\vartheta_{\mathrm m})
&=\frac12Q_n\bigl(\mathcal S_n(h)\bigr)-\frac14\tr\bigl(\mathcal S_n(h)^2\bigr) \label{eq:mfbm-lan-expand}\\
&\quad-\frac12\Bigl(Z_n^\top\mathcal S_n(h)^2Z_n-\tr\bigl(\mathcal S_n(h)^2\bigr)\Bigr)
-\frac12R_{\log}\bigl(\mathcal S_n(h)\bigr)\nonumber\\
&\quad-\frac12 Z_n^\top R_{\mathrm{inv}}\bigl(\mathcal S_n(h)\bigr)Z_n.
\end{align}
Now write $\mathcal S_n(h)=B_n^{\mathrm m}(h)+\mathcal R_n^{\mathrm m}(h)$. Since
\[
\frac12Q_n\bigl(B_n^{\mathrm m}(h)\bigr)=h^\top\Xi_n^{\mathrm m}
\]
and
\[
\Var\!\left(\frac12Q_n\bigl(\mathcal R_n^{\mathrm m}(h)\bigr)\right)
=\frac12\tr\Bigl(\bigl(\mathcal R_n^{\mathrm m}(h)\bigr)^2\Bigr)\to0,
\]
we have
\begin{equation}\label{eq:mfbm-linear-term}
\frac12Q_n\bigl(\mathcal S_n(h)\bigr)=h^\top\Xi_n^{\mathrm m}+o_{P_{\vartheta_{\mathrm m}}}(1).
\end{equation}
For the quadratic term,
\[
\tr\bigl(\mathcal S_n(h)^2\bigr)
=
\tr\bigl(B_n^{\mathrm m}(h)^2\bigr)
+2\tr\bigl(B_n^{\mathrm m}(h)\mathcal R_n^{\mathrm m}(h)\bigr)
+\tr\bigl((\mathcal R_n^{\mathrm m}(h))^2\bigr).
\]
Because $\tr(B_n^{\mathrm m}(h)^2)=O(1)$ by proposition \ref{prop:mfbm-traces}, the Cauchy--Schwarz inequality and the local linearization result yield
\[
\tr\bigl(\mathcal S_n(h)^2\bigr)=\tr\bigl(B_n^{\mathrm m}(h)^2\bigr)+o(1).
\]
Applying proposition \ref{prop:mfbm-traces} once more gives
\begin{equation}\label{eq:mfbm-quadratic-term}
\frac14\tr\bigl(\mathcal S_n(h)^2\bigr)
=\frac14\tr\bigl(B_n^{\mathrm m}(h)^2\bigr)+o(1)
=\frac12 h^\top\Gamma_{\mathrm m}^{\mathrm{crit}}h+o(1).
\end{equation}
Moreover,
\[
\Var\Bigl(Z_n^\top\mathcal S_n(h)^2Z_n-\tr\bigl(\mathcal S_n(h)^2\bigr)\Bigr)
=2\tr\bigl(\mathcal S_n(h)^4\bigr)
\le 2\norm{\mathcal S_n(h)}_{\op}^2\tr\bigl(\mathcal S_n(h)^2\bigr)\to0,
\]
so
\begin{equation}\label{eq:mfbm-quadratic-fluct}
Z_n^\top\mathcal S_n(h)^2Z_n-\tr\bigl(\mathcal S_n(h)^2\bigr)=o_P(1).
\end{equation}
Next, lemma \ref{lem:matrix-taylor} and the previous bounds imply
\begin{equation}\label{eq:mfbm-log-rem}
R_{\log}\bigl(\mathcal S_n(h)\bigr)
=O\!\left(\norm{\mathcal S_n(h)}_{\op}\tr\bigl(\mathcal S_n(h)^2\bigr)\right)
=o(1).
\end{equation}
Also,
\[
\tr\Bigl(R_{\mathrm{inv}}\bigl(\mathcal S_n(h)\bigr)^2\Bigr)
\le C\tr\bigl(\mathcal S_n(h)^6\bigr)
\le C\norm{\mathcal S_n(h)}_{\op}^4\tr\bigl(\mathcal S_n(h)^2\bigr)\to0,
\]
while
\[
\abs{\tr\Bigl(R_{\mathrm{inv}}\bigl(\mathcal S_n(h)\bigr)\Bigr)}
\le C\norm{\mathcal S_n(h)}_{\op}\tr\bigl(\mathcal S_n(h)^2\bigr)=o(1).
\]
Therefore centered Gaussian quadratic-form bounds give
\begin{equation}\label{eq:mfbm-inv-rem}
Z_n^\top R_{\mathrm{inv}}\bigl(\mathcal S_n(h)\bigr)Z_n=o_P(1).
\end{equation}
Combining \eqref{eq:mfbm-lan-expand}--\eqref{eq:mfbm-inv-rem} yields
\[
\ell_n^{\mathrm m}(\vartheta_{\mathrm m,n}(h))-
\ell_n^{\mathrm m}(\vartheta_{\mathrm m})
=
h^\top\Xi_n^{\mathrm m}-\frac12 h^\top\Gamma_{\mathrm m}^{\mathrm{crit}}h+o_{P_{\vartheta_{\mathrm m}}}(1).
\]
Finally, theorem \ref{thm:mfbm-main} gives $\Xi_n^{\mathrm m}\dto \mathcal{N}(0,\Gamma_{\mathrm m}^{\mathrm{crit}})$, so the mfBm experiment is LAN at $\vartheta_{\mathrm m}$.
\end{proof}

\begin{proof}[\textbf{Proof of theorem \ref{thm:mfou-lan}}]
Fix $h\in\R^3$ and write
\[
\Sigma_n:=\Sigma_n^{\mathrm o}(\vartheta_{\mathrm o}),
\qquad
\Sigma_{n,h}:=\Sigma_n^{\mathrm o}(\vartheta_{\mathrm o,n}(h)),
\qquad
Z_n:=\Sigma_n^{-1/2}X_n^{\mathrm o}\sim \mathcal{N}(0,I_n).
\]
Define $\mathcal S_n(h)=\Sigma_n^{-1/2}(\Sigma_{n,h}-\Sigma_n)\Sigma_n^{-1/2}$. As in \eqref{eq:mfbm-likelihood-identity},
\[
\ell_n^{\mathrm o}(\vartheta_{\mathrm o,n}(h))-
\ell_n^{\mathrm o}(\vartheta_{\mathrm o})
=
-\frac12\log\det\bigl(I_n+\mathcal S_n(h)\bigr)
-\frac12 Z_n^\top\Bigl((I_n+\mathcal S_n(h))^{-1}-I_n\Bigr)Z_n.
\]
By proposition \ref{prop:critical-linearization}, corollary \ref{cor:cubic-remainder} and lemma \ref{lem:matrix-taylor},
\begin{align}\label{eq:mfou-lan-expand}
\ell_n^{\mathrm o}(\vartheta_{\mathrm o,n}(h))-
\ell_n^{\mathrm o}(\vartheta_{\mathrm o})
&=\frac12Q_n\bigl(\mathcal S_n(h)\bigr)-\frac14\tr\bigl(\mathcal S_n(h)^2\bigr)\nonumber\\
&\quad-\frac12\Bigl(Z_n^\top\mathcal S_n(h)^2Z_n-\tr\bigl(\mathcal S_n(h)^2\bigr)\Bigr)\nonumber\\
&\quad-\frac12R_{\log}\bigl(\mathcal S_n(h)\bigr)
-\frac12 Z_n^\top R_{\mathrm{inv}}\bigl(\mathcal S_n(h)\bigr)Z_n.
\end{align}
Write $\mathcal S_n(h)=B_n^{\mathrm o}(h)+\mathcal R_n^{\mathrm o}(h)$. Since
\[
\frac12Q_n\bigl(B_n^{\mathrm o}(h)\bigr)=h^\top\Xi_n^{\mathrm o},
\qquad
\tr\Bigl(\bigl(\mathcal R_n^{\mathrm o}(h)\bigr)^2\Bigr)\to0,
\]
we obtain
\begin{equation}\label{eq:mfou-linear-term}
\frac12Q_n\bigl(\mathcal S_n(h)\bigr)=h^\top\Xi_n^{\mathrm o}+o_{P_{\vartheta_{\mathrm o}}}(1).
\end{equation}
Furthermore,
\[
\tr\bigl(\mathcal S_n(h)^2\bigr)=\tr\bigl(B_n^{\mathrm o}(h)^2\bigr)+o(1).
\]
Now proposition \ref{prop:mfou-traces} gives
\[
\frac14\tr\bigl(B_n^{\mathrm o}(h)^2\bigr)
=\frac12 h^\top\Gamma_{\mathrm o}^{\mathrm{crit}}h+o(1),
\]
because the mixed traces with $\Aon$ are negligible under the critical normalization exactly as in the proof of theorem \ref{thm:mfou-main}. Hence
\begin{equation}\label{eq:mfou-quadratic-term}
\frac14\tr\bigl(\mathcal S_n(h)^2\bigr)=\frac12 h^\top\Gamma_{\mathrm o}^{\mathrm{crit}}h+o(1).
\end{equation}
The higher-order terms are controlled exactly as in the mfBm case since $\norm{\mathcal S_n(h)}_{\op}\to0$ and $\tr(\mathcal S_n(h)^2)=O(1)$:
\[
Z_n^\top\mathcal S_n(h)^2Z_n-\tr\bigl(\mathcal S_n(h)^2\bigr)=o_P(1),
\qquad
R_{\log}\bigl(\mathcal S_n(h)\bigr)=o(1),
\qquad
Z_n^\top R_{\mathrm{inv}}\bigl(\mathcal S_n(h)\bigr)Z_n=o_P(1).
\]
Combining these estimates with \eqref{eq:mfou-lan-expand}--\eqref{eq:mfou-quadratic-term} yields
\[
\ell_n^{\mathrm o}(\vartheta_{\mathrm o,n}(h))-
\ell_n^{\mathrm o}(\vartheta_{\mathrm o})
=
h^\top\Xi_n^{\mathrm o}-\frac12 h^\top\Gamma_{\mathrm o}^{\mathrm{crit}}h+o_{P_{\vartheta_{\mathrm o}}}(1).
\]
Finally, theorem \ref{thm:mfou-main} implies $\Xi_n^{\mathrm o}\dto \mathcal{N}(0,\Gamma_{\mathrm o}^{\mathrm{crit}})$, so the mfOU experiment is LAN at the critical point.
\end{proof}

\begin{remark}
The limiting information matrices $\Gamma_{\mathrm m}^{\mathrm{crit}}$ and $\Gamma_{\mathrm o}^{\mathrm{crit}}$ are deterministic and non-singular. Hence the critical experiments are genuinely LAN, not merely LAMN. In particular, the critical point is resolved by a single triangular local reparametrization.\end{remark}
\section{Boundary score tests and feasible implementation}\label{sec:score-tests}
The LAN expansions obtained above are useful only if they can be turned into statistics whose meaning is clear at the boundary. At $H=3/4$, this means more than writing down a standardized central sequence. The null hypothesis is composite on the left and nuisance estimation cannot be avoided in any serious application. What the critical geometry gives us is a particularly transparent solution to the first of these problems.

Once the score is expressed in the transformed coordinates of Section~\ref{sec:preliminaries}, the $H$-direction can be separated from the nuisance tangent space in exactly the same way as in regular Gaussian LAN problems, except that the relevant normalization now carries the logarithmic critical rate.\\
Throughout this section,
\[
z_p:=\Phi^{-1}(p)
\]
denotes the $p$-quantile of the standard normal law. Because the local experiment is centered at the boundary itself, all critical values are calibrated at $H=3/4$. The econometric object of interest is therefore the $H$-specific part of the critical central sequence after removing the contribution of the nuisance parameter $\sigma$ in the mfBm model, and of the nuisance parameters $\sigma$ and $\alpha$ in the mfOU model.

In the mfBm experiment this construction is especially direct, because the only nuisance parameter is the scale $\sigma$. Let
\[
\Xi_n^{\mathrm m}
=
\begin{pmatrix}
\Xi_{\sigma,n}^{\mathrm m}\\[0.3em]
\Xi_{H,n}^{\mathrm m}
\end{pmatrix}
\]
be the normalized score vector in theorem \ref{thm:mfbm-lan}, and write
\[
\Gamma_{\mathrm m}^{\mathrm{crit}}
=
\begin{pmatrix}
\Gamma_{\sigma\sigma}^{\mathrm m} & \Gamma_{\sigma H}^{\mathrm m}\\
\Gamma_{H\sigma}^{\mathrm m} & \Gamma_{HH}^{\mathrm m}
\end{pmatrix}.
\]
Projecting the $H$-component away from the nuisance direction yields the efficient central sequence
\[
\Delta_{H,n}^{\mathrm{eff,m}}
:=
\Xi_{H,n}^{\mathrm m}
-
\Gamma_{H\sigma}^{\mathrm m}
\bigl(\Gamma_{\sigma\sigma}^{\mathrm m}\bigr)^{-1}
\Xi_{\sigma,n}^{\mathrm m}
=
\Xi_{H,n}^{\mathrm m}-\frac{\sigma}{2}\Xi_{\sigma,n}^{\mathrm m},
\]
with efficient information
\[
I_H^{\mathrm{eff,m}}
:=
\Gamma_{HH}^{\mathrm m}
-
\Gamma_{H\sigma}^{\mathrm m}
\bigl(\Gamma_{\sigma\sigma}^{\mathrm m}\bigr)^{-1}
\Gamma_{\sigma H}^{\mathrm m}
=
\frac{3\sigma^4}{64}.
\]
Accordingly, the natural boundary statistic is
\[
T_n^{\mathrm m}
:=
\frac{\Delta_{H,n}^{\mathrm{eff,m}}}{\sqrt{I_H^{\mathrm{eff,m}}}}
=
\frac{8}{\sqrt3\,\sigma^2}
\left(
\Xi_{H,n}^{\mathrm m}-\frac{\sigma}{2}\Xi_{\sigma,n}^{\mathrm m}
\right).
\]
The formula is simple, but its interpretation matters. The statistic does not test whether the raw $H$-score is large in isolation. It tests whether the component of that score which remains after removing the critical scale effect points into the supercritical region.

The mfOU experiment is formally richer but conceptually no different. There the nuisance parameter is two-dimensional, with $\eta:=(\sigma,\alpha)$, and one might expect the additional drift direction to complicate the efficient score. The critical information matrix shows that this does not happen at the leading order. If
\[
\Xi_n^{\mathrm o}
=
\begin{pmatrix}
\Xi_{\sigma,n}^{\mathrm o}\\[0.3em]
\Xi_{H,n}^{\mathrm o}\\[0.3em]
\Xi_{\alpha,n}^{\mathrm o}
\end{pmatrix}
\]
denotes the normalized score vector in theorem \ref{thm:mfou-lan}, and if we write
\[
\Xi_{\eta,n}^{\mathrm o}
:=
\begin{pmatrix}
\Xi_{\sigma,n}^{\mathrm o}\\[0.3em]
\Xi_{\alpha,n}^{\mathrm o}
\end{pmatrix},
\qquad
\Gamma_{\mathrm o}^{\mathrm{crit}}
=
\begin{pmatrix}
\Gamma_{\eta\eta}^{\mathrm o} & \Gamma_{\eta H}^{\mathrm o}\\
\Gamma_{H\eta}^{\mathrm o} & \Gamma_{HH}^{\mathrm o}
\end{pmatrix},
\]
then the efficient $H$-score is
\[
\Delta_{H,n}^{\mathrm{eff,o}}
:=
\Xi_{H,n}^{\mathrm o}
-
\Gamma_{H\eta}^{\mathrm o}
\bigl(\Gamma_{\eta\eta}^{\mathrm o}\bigr)^{-1}
\Xi_{\eta,n}^{\mathrm o}.
\]
Because \eqref{eq:gamma-o-main} gives $\Gamma_{H\alpha}^{\mathrm o}=0$, the drift parameter falls out of the leading efficient projection and one obtains the same critical direction as in mfBm,
\[
\Delta_{H,n}^{\mathrm{eff,o}}
=
\Xi_{H,n}^{\mathrm o}-\frac{\sigma}{2}\Xi_{\sigma,n}^{\mathrm o},
\qquad
I_H^{\mathrm{eff,o}}
=
\frac{3\sigma^4}{64},
\]
so that
\[
T_n^{\mathrm o}
:=
\frac{\Delta_{H,n}^{\mathrm{eff,o}}}{\sqrt{I_H^{\mathrm{eff,o}}}}
=
\frac{8}{\sqrt3\,\sigma^2}
\left(
\Xi_{H,n}^{\mathrm o}-\frac{\sigma}{2}\Xi_{\sigma,n}^{\mathrm o}
\right).
\]
This coincidence is not merely algebraic convenience. It shows that the critical inferential problem is driven by the same transformed $(\sigma,H)$ block in both models, while the additional OU drift enters only beyond the leading boundary geometry.

\begin{corollary}
\label{cor:one-sided-tests}
Assume the critical LAN expansions in theorems \ref{thm:mfbm-lan} and \ref{thm:mfou-lan}. Under the critical null hypothesis $H=3/4$,
\[
T_n^{\mathrm m}\dto \mathcal{N}(0,1),
\qquad
T_n^{\mathrm o}\dto \mathcal{N}(0,1).
\]
Under local alternatives with positive $H$-component, namely $h=(\cdot,h_H)$ in the mfBm case and $h=(\cdot,h_H,\cdot)$ in the mfOU case with $h_H>0$,
\[
T_n^{\mathrm m}\dto \mathcal{N}\!\bigl(h_H\sqrt{I_H^{\mathrm{eff,m}}},1\bigr),
\qquad
T_n^{\mathrm o}\dto \mathcal{N}\!\bigl(h_H\sqrt{I_H^{\mathrm{eff,o}}},1\bigr).
\]
Consequently, an asymptotic level-$\alpha$ boundary-calibrated one-sided test of
\[
H_0:\ H\le \frac34
\qquad\text{vs}\qquad
H_1:\ H>\frac34
\]
rejects for large values of the statistic, that is, whenever
\[
T_n^{\mathrm m}>z_{1-\alpha}
\quad\text{or}\quad
T_n^{\mathrm o}>z_{1-\alpha},
\]
according to the model under consideration.
\end{corollary}

\begin{proof}
Under the critical null, the LAN central sequences in theorems~\ref{thm:mfbm-lan} and~\ref{thm:mfou-lan} converge to centered Gaussian limits with covariance matrices $\Gamma_{\mathrm m}^{\mathrm{crit}}$ and $\Gamma_{\mathrm o}^{\mathrm{crit}}$. In the mfBm experiment, efficient projection removes the scale direction from the $H$-component of the central sequence and leaves
\[
\Delta_{H,n}^{\mathrm{eff,m}}
=
\Xi^{\mathrm m}_{H,n}-\frac{\sigma}{2}\Xi^{\mathrm m}_{\sigma,n}.
\]
Its asymptotic variance is the Schur complement
\[
I_{H}^{\mathrm{eff,m}}
=
\Gamma^{\mathrm m}_{HH}
-
\Gamma^{\mathrm m}_{H\sigma}
\bigl(\Gamma^{\mathrm m}_{\sigma\sigma}\bigr)^{-1}
\Gamma^{\mathrm m}_{\sigma H},
\]
so standardization yields $T_n^{\mathrm m}\dto\mathcal N(0,1)$.

The mfOU case follows the same argument, except that the nuisance vector is $\eta=(\sigma,\alpha)$. Because the leading critical information satisfies $\Gamma^{\mathrm o}_{H\alpha}=0$, the efficient projection again reduces to
\[
\Delta_{H,n}^{\mathrm{eff,o}}
=
\Xi^{\mathrm o}_{H,n}-\frac{\sigma}{2}\Xi^{\mathrm o}_{\sigma,n},
\]
and its variance is the corresponding Schur complement $I_H^{\mathrm{eff,o}}$. This gives $T_n^{\mathrm o}\dto\mathcal N(0,1)$ under the boundary null.

Under local alternatives, Le Cam's third lemma (\cite{vdV1998}) transports the LAN drift into the efficient score. A positive $H$-component therefore shifts the mean of the efficient statistic to the right while leaving the asymptotic variance unchanged, yielding the limits
\[
T_n^m \xrightarrow{d} N\!\bigl(h_H\sqrt{I_H^{\mathrm{eff},m}},1\bigr),
\qquad
T_n^o \xrightarrow{d} N\!\bigl(h_H\sqrt{I_H^{\mathrm{eff},o}},1\bigr).
\]
The one-sided rejection region follows immediately from this directional mean shift.
\end{proof}

The one-sided version is the economically relevant test throughout the paper, because the question of interest is whether the data have crossed into the supercritical region. A two-sided statistic can be written down just as easily, but it addresses a different question. We record it only for completeness.

\begin{corollary}
\label{cor:two-sided-tests}
Under the assumptions of corollary \ref{cor:one-sided-tests}, an asymptotic
level-$\alpha$ two-sided test of
\[
H_0:\ H=\frac34
\qquad\text{vs}\qquad
H_1:\ H\neq \frac34
\]
rejects whenever
\[
|T_n^{\mathrm m}|>z_{1-\alpha/2}
\quad\text{or}\quad
|T_n^{\mathrm o}|>z_{1-\alpha/2},
\]
according to the model under consideration.
\end{corollary}

The oracle statistics $T_n^{\bullet}$, $\bullet\in\{\mathrm m,\mathrm o\}$, obtained when $\sigma$ is known, are the right theoretical benchmarks, but they are not yet the objects that an applied researcher would compute. In practice one would estimate nuisance parameters under the boundary restriction $H=3/4$, evaluate the transformed score at that restricted estimator, and then standardize the efficient projection with the corresponding plug-in information. For mfBm this means replacing $\sigma$ by a restricted estimator $\widehat\sigma_0$, for mfOU one analogously works with $(\widehat\sigma_0,\widehat\alpha_0)$. The resulting statistic is the feasible boundary score.

This feasible version is where the composite null issue becomes genuinely econometric. The LAN calculations already show that the critical score itself has a clean nuisance orthogonal representation. What remains is to justify that estimating the nuisance parameters under the restriction $H=3/4$ does not feed back into the first-order null law. Put differently, one needs a result showing that the plug-in error enters only at a smaller order than the logarithmically normalized efficient score. In regular likelihood settings such a statement is often routine, here it is precisely the sort of point that deserves separate treatment because the critical rate is slow and the boundary experiment is singular before transformation.

\begin{lemma}
\label{lem:plugin-validity-mfbm}
Assume that under the boundary null $H=3/4$ the restricted mfBm estimator satisfies
\[
\widehat\sigma_0-\sigma = O_P\bigl((n\Delta_n)^{-1/2}\bigr),
\]
and that there exists a deterministic sequence $\varepsilon_n\downarrow0$ such that
\[
\sup_{|\widetilde\sigma-\sigma|\le \varepsilon_n}
\left|\partial_\sigma \widehat T_n^{\mathrm m}(\widetilde\sigma)\right|
=O_P\bigl((n\Delta_n)^{1/2}L_n^{-3/2}\bigr).
\]
Then
\[
\widehat T_n^{\mathrm m}(\widehat\sigma_0)=\widehat T_n^{\mathrm m}(\sigma)+o_P(1).
\]
In particular, once the oracle statistic has the standard normal null limit, the feasible mfBm statistic has the same first-order null law.
\end{lemma}

\begin{proof}
With probability tending to one, $|\widehat\sigma_0-\sigma|\le \varepsilon_n$, so the mean-value theorem yields
\[
\widehat T_n^{\mathrm m}(\widehat\sigma_0)-\widehat T_n^{\mathrm m}(\sigma)
=(\widehat\sigma_0-\sigma)\,\partial_\sigma \widehat T_n^{\mathrm m}(\widetilde\sigma_n)
\]
for some random $\widetilde\sigma_n$ between $\widehat\sigma_0$ and $\sigma$. The first factor is $O_P((n\Delta_n)^{-1/2})$ by assumption, whereas the second is uniformly $O_P((n\Delta_n)^{1/2}L_n^{-3/2})$ on that neighborhood. Their product is therefore $O_P(L_n^{-3/2})=o_P(1)$, which proves the claim.
\end{proof}

This lemma gives the minimal condition needed to complete the feasible mfBm argument. Once the restricted estimator has the usual boundary rate and the feasible score is locally smooth in $\sigma$, the slow logarithmic normalization pushes the plug-in contribution below the first-order null scale. The mfOU case is handled in an analogous manner. There the nuisance vector is $(\sigma,\alpha)$ rather than $\sigma$ alone, but the same Taylor expansion applies componentwise and no new first-order term is created by the drift direction because the critical $(H,\alpha)$ block has already been projected out. This is why the Monte Carlo analysis below focuses on the feasible mfBm statistic rather than on the oracle object. The numerical question is not whether the LAN formulas can be standardized on paper, but whether the plug-in score already behaves like a practical diagnostic for entry into the supercritical regime.

The same construction can be turned around if one wishes to study movement below the boundary rather than above it. In that case the rejection region is placed in the left tail, that is, $T_n^{\bullet}<z_{\alpha}$ for $\bullet\in\{\mathrm m,\mathrm o\}$.

\section{Monte Carlo diagnostics for the feasible boundary statistic}\label{sec:simulation}
The numerical question at the boundary is not whether the transformed score has the right sign on paper, but whether the feasible statistic can already be computed in a stable way and how conservative its null calibration remains once the logarithmic factor $L_n=\log(1/\Delta_n)$ enters the normalization. We consider two mfBm experiments with 1000 replications each, both built around the feasible one-sided statistic.
\[
\widehat T_n^{\mathrm m}
=
\frac{8}{\sqrt3\,\widehat{\sigma}_0^2}
\left(
\widehat{\Xi}_{H,n}^{\mathrm m}
-\frac{\widehat{\sigma}_0}{2}\widehat{\Xi}_{\sigma,n}^{\mathrm m}
\right),
\]
where the nuisance parameter is estimated under the boundary restriction
\[
\widehat{\sigma}_0
:=
\arg\max_{\sigma>0}\ell_n^{\mathrm m}(\sigma,3/4).
\]
In every replication we simulate the mfBm increment vector under the relevant data-generating law, re-estimate $\widehat{\sigma}_0$ under $H=3/4$, and then evaluate $\widehat T_n^{\mathrm m}$ exactly as in Section~\ref{sec:score-tests}. Large positive values support $H>3/4$.

The first experiment studies directional power at a fixed high-frequency design,
\[
n=160,\qquad \Delta_n=10^{-2},\qquad \sigma=3,
\]
over the grid
\[
H\in\{0.75,0.78,0.80,0.85,0.90\}.
\]
The second experiment keeps the model at the boundary,
\[
H=3/4,\qquad \sigma=3,
\]
and follows the asymptotic sequence
\[
\Delta_n=n^{-1/2},
\qquad
n\in\{64,128,256,512\},
\]
again with $1{,}000$ replications at each design point.

\begin{table}[t]
\centering
\caption{Experiment I: power curve of the feasible one-sided mfBm statistic at a fixed design}
\label{tab:sim-power-grid}
\begin{tabular}{ccccccc}
\hline
True $H$ & mean$(\widehat T_n^{\mathrm m})$ & sd$(\widehat T_n^{\mathrm m})$ & rej. 10\% & rej. 5\% & mean$(\widehat\sigma_0)$ & sd$(\widehat\sigma_0)$\\
\hline
$0.75$ & $0.013$ & $0.528$ & $0.022$ & $0.009$ & $2.982$ & $0.397$\\
$0.78$ & $0.190$ & $0.761$ & $0.074$ & $0.038$ & $2.540$ & $0.392$\\
$0.80$ & $0.386$ & $0.854$ & $0.134$ & $0.080$ & $2.339$ & $0.392$\\
$0.85$ & $0.854$ & $1.375$ & $0.321$ & $0.245$ & $1.839$ & $0.444$\\
$0.90$ & $1.422$ & $1.846$ & $0.474$ & $0.394$ & $1.472$ & $0.494$\\
\hline
\end{tabular}

\end{table}

\begin{table}[t]
\centering
\caption{Experiment II: critical-null behavior of the feasible mfBm statistic along an asymptotic sequence}
\label{tab:sim-null-sequence}
\begin{tabular}{cccccccc}
\hline
$n$ & $\Delta_n$ & mean$(\widehat T_n^{\mathrm m})$ & sd$(\widehat T_n^{\mathrm m})$ & rej. 10\% & rej. 5\% & mean$(\widehat\sigma_0)$ & sd$(\widehat\sigma_0)$\\
\hline
$64$ & $0.1250$ & $0.018$ & $0.688$ & $0.044$ & $0.023$ & $2.990$ & $0.378$\\
$128$ & $0.0884$ & $-0.013$ & $0.593$ & $0.029$ & $0.009$ & $2.973$ & $0.283$\\
$256$ & $0.0625$ & $0.011$ & $0.549$ & $0.026$ & $0.006$ & $2.995$ & $0.211$\\
$512$ & $0.0442$ & $-0.034$ & $0.511$ & $0.007$ & $0.002$ & $2.995$ & $0.152$\\
\hline
\end{tabular}

\end{table}

Table~\ref{tab:sim-power-grid} shows that the feasible statistic is directionally informative in the way the theory predicts. Under the boundary law it remains essentially centered and heavily conservative, but its mean and rejection frequency both increase as the true Hurst index moves deeper into the supercritical region. The power gain is gradual near $H=0.78$ and $H=0.80$, then becomes much more visible at $H=0.85$ and $H=0.90$.

Table~\ref{tab:sim-null-sequence} shows equally clearly that the conservative null behavior is not a small-sample accident. Even as $n$ increases from $64$ to $512$, the empirical standard deviation of $\widehat T_n^{\mathrm m}$ stays far below the asymptotic unit benchmark and the rejection rates at nominal $10\%$ and $5\%$ levels remain well below target. The statistic is therefore numerically stable and already useful as a one-sided diagnostic, but first-order Gaussian critical values still under-reject substantially at realistic sample sizes.

Figure~\ref{fig:mc-power} and Figure~\ref{fig:mc-null} summarize the same two messages graphically. The first is that the feasible score has a genuine power curve once the data move above the boundary. The second is that the critical logarithmic normalization makes convergence slow, so we should eventually supplement the asymptotic critical values with a finite-sample correction such as a null-based bootstrap.

\begin{figure}[h]
\centering
\includegraphics[width=0.72\textwidth]{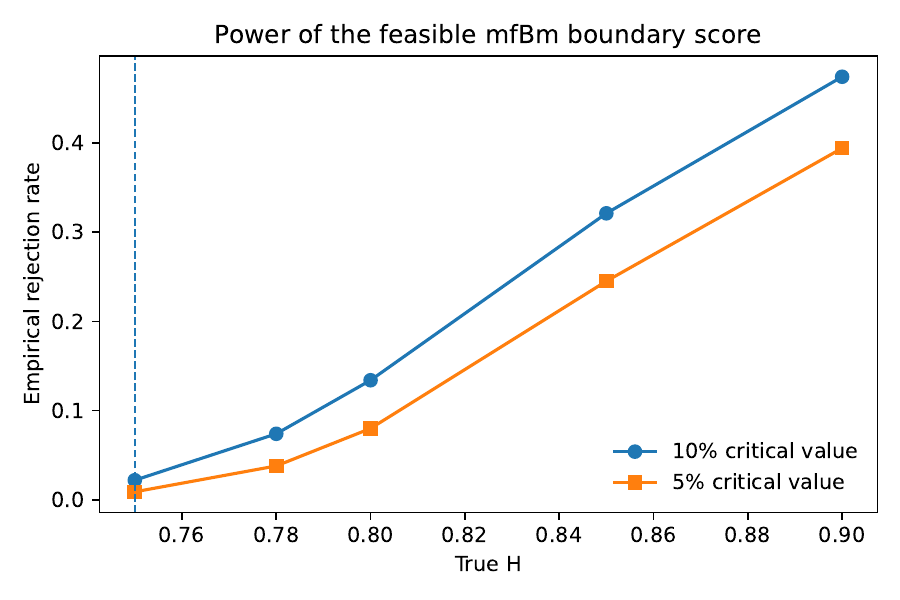}
\caption{Empirical rejection rates of the feasible mfBm boundary score across the supercritical power grid in Experiment~I.}
\label{fig:mc-power}
\end{figure}

\begin{figure}[h]
\centering
\includegraphics[width=0.72\textwidth]{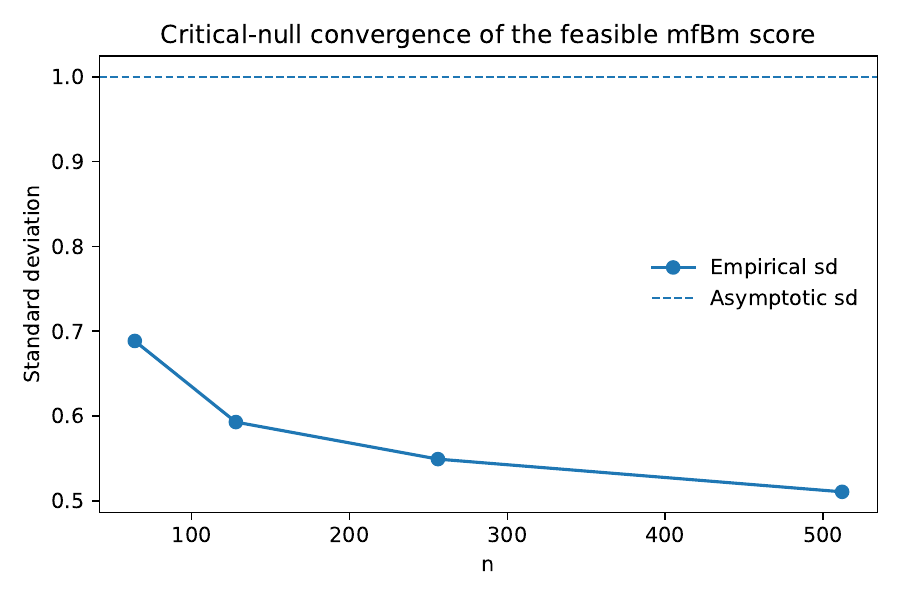}
\caption{Empirical standard deviation of the feasible mfBm boundary score under the critical null across the asymptotic sequence in Experiment~II.}
\label{fig:mc-null}
\end{figure}

 The Monte Carlo section is not definitive, but the upgraded simulations already show that the feasible mfBm statistic is stable, its power increases as expected, and its main shortcoming is conservative null calibration, not numerical instability.

\section{Empirical implementation of the exact feasible boundary score}\label{sec:empirical}
We apply the exact feasible mfBm statistic to one-minute SPY data from 2008-01-22 to 2021-05-06. Because overnight gaps break the equally spaced increment scheme, we work day by day on the regular session only, keeping observations from 07{:}30 to 13{:}59 and discarding incomplete days. This leaves $2{,}621$ trading days, each with $390$ prices and therefore $n=389$ close-to-close log returns.

For each day we first normalize the intraday return vector so that its empirical diffusive variance matches the Brownian benchmark $\Delta_n=1/389$. Under the boundary restriction $H=3/4$ we then estimate $\sigma$ by the exact Gaussian likelihood,
\[
\widehat\sigma_0:=\arg\max_{\sigma>0}\,\ell_n^{\mathrm m}(\sigma,3/4),
\]
and evaluate
\[
\widehat\Xi_{\sigma,n}^{\mathrm m}:=\frac{S_{\sigma,n}^{\mathrm m}(\widehat\sigma_0,3/4)}{\sqrt{n\Delta_nL_n}},
\qquad
\widehat\Xi_{H,n}^{\mathrm m}:=\frac{R_{H,n}^{\mathrm m}(\widehat\sigma_0,3/4)}{\sqrt{n\Delta_n}\,L_n^{3/2}}.
\]
The empirical boundary statistic is therefore
\begin{equation}\label{eq:empirical-feasible-score}
\widehat T_n^{\mathrm m}
:=\frac{8}{\sqrt3\,\widehat\sigma_0^2}
\left(
\widehat\Xi_{H,n}^{\mathrm m}-\frac{\widehat\sigma_0}{2}\widehat\Xi_{\sigma,n}^{\mathrm m}
\right).
\end{equation}
Large positive values support the one-sided alternative $H>3/4$.

Table~\ref{tab:spy-full} summarizes the daily statistic over the full sample. The center of the distribution stays close to zero, the empirical dispersion is slightly below one, and the rejection shares are modest: $10.4\%$ at the nominal $10\%$ threshold and $6.5\%$ at the nominal $5\%$ threshold. In other words, the exact feasible score can be computed stably on genuine intraday data, but the SPY sample does not display persistent evidence for a supercritical regime.

\begin{table}[h]
\centering
\caption{One-minute SPY boundary score: full-sample daily summary}
\label{tab:spy-full}
\begin{tabular}{lr}
\toprule
Statistic & Value \\
\midrule
Trading days & 2621 \\
Mean $\widehat T_n^{\mathrm m}$ & 0.005 \\
Median $\widehat T_n^{\mathrm m}$ & -0.185 \\
Std. dev. $\widehat T_n^{\mathrm m}$ & 0.916 \\
Rejection share (10\%) & 10.4\% \\
Rejection share (5\%) & 6.5\% \\
Median $\widehat\sigma_0$ & 0.0002 \\
\bottomrule
\end{tabular}
\end{table}

Table~\ref{tab:spy-sub} breaks the sample into broad subperiods. The 2008--2009 block is somewhat more positive on average, while the 2020--2021 block is mildly negative, but none of the subperiod summaries suggests a sustained right-tail shift large enough to indicate repeated entry into $H>3/4$.

\begin{table}[h]
\centering
\caption{One-minute SPY boundary score: subperiod summaries}
\label{tab:spy-sub}
\begin{tabular}{lrrrrrr}
\toprule
Period & Days & Mean $T$ & Median $T$ & Reject 10\% & Reject 5\% & Median $\widehat{\sigma}_0$ \\
\midrule
2008-2009 & 385 & 0.118 & -0.091 & 13.0\% & 7.3\% & 0.0002 \\
2010-2019 & 1968 & 0.006 & -0.188 & 10.3\% & 6.6\% & 0.0002 \\
2020-2021 & 268 & -0.158 & -0.301 & 7.5\% & 4.9\% & 0.0002 \\
\bottomrule
\end{tabular}
\end{table}

Figure~\ref{fig:spy-daily} plots the daily statistic together with its 60-day rolling mean. Crossings of the one sided critical lines occur, but they are not persistent and do not cumulate into a persistent upward regime. Figure~\ref{fig:spy-rollrej} reports the 60-day rolling rejection share at the $5\%$ level. It fluctuates around the nominal benchmark and never points to a long interval of systematic supercritical rejection.

\begin{figure}[t]
\centering
\includegraphics[width=0.76\textwidth]{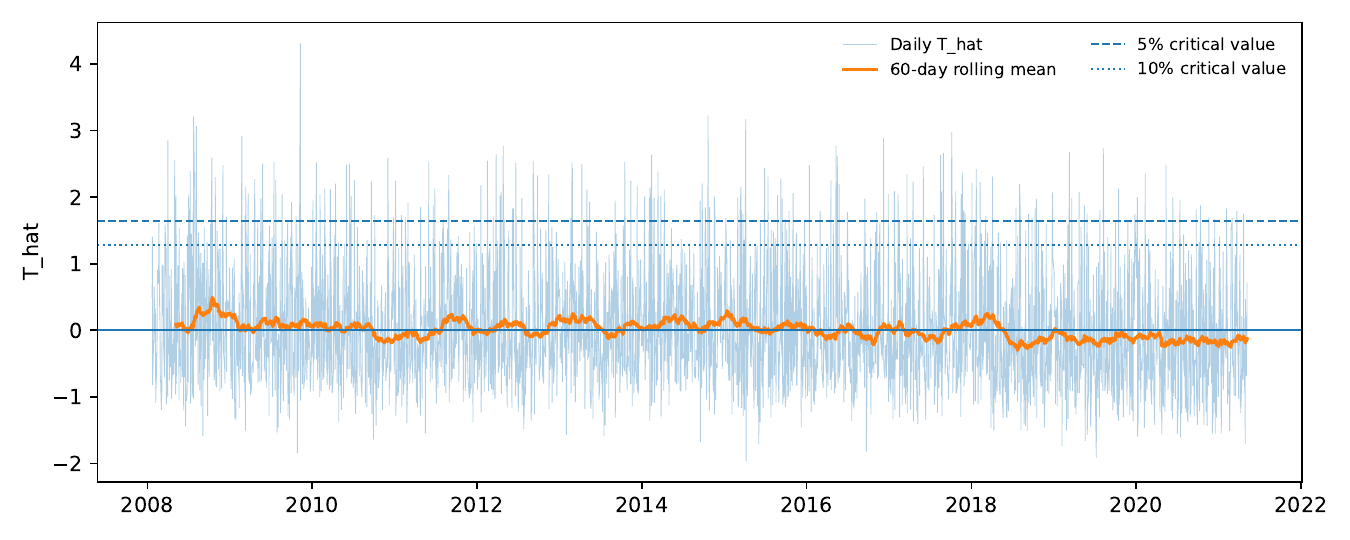}
\caption{Daily exact feasible mfBm statistic on one-minute SPY returns, with 60-day rolling mean and one-sided critical lines.}
\label{fig:spy-daily}
\end{figure}

\begin{figure}[t]
\centering
\includegraphics[width=0.76\textwidth]{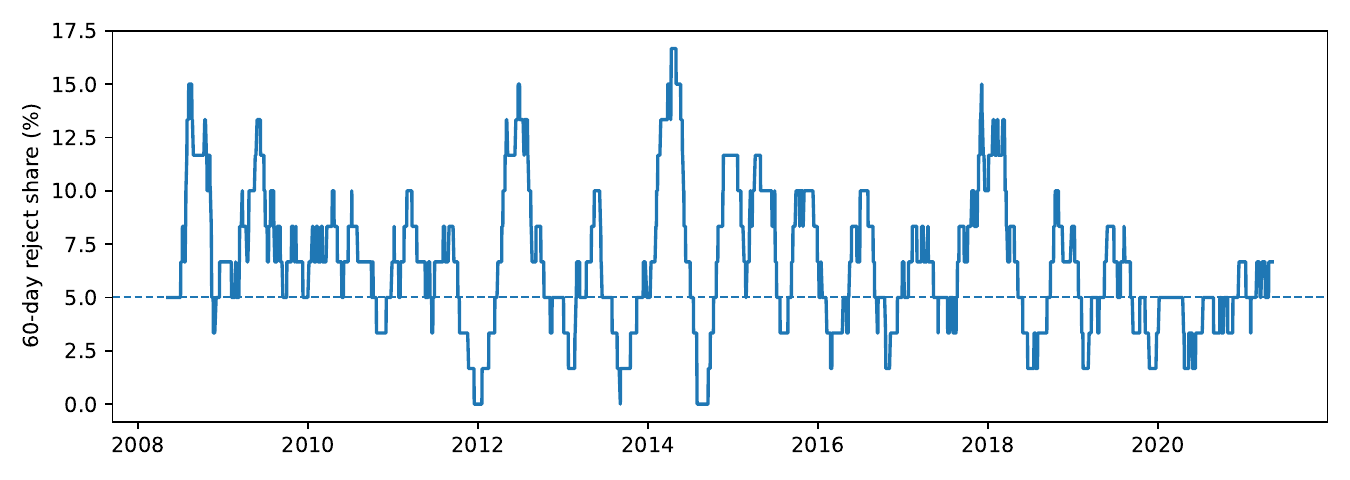}
\caption{60-day rolling rejection share of the exact feasible mfBm statistic at the nominal $5\%$ level.}
\label{fig:spy-rollrej}
\end{figure}

The empirical conclusion is therefore narrow but concrete. Once the section is built around the statistic developed earlier in the paper, the main message is not a recycled roughness estimate but a direct boundary statement: in this one-minute SPY sample, the exact feasible score does not deliver persistent evidence that the data live on the supercritical side of the semimartingale threshold.

This is the right interpretation of the empirical exercise. The object of interest in the present paper is not whether the data exhibit roughness in the now standard sense of very small Hurst exponents, but whether they cross the much higher boundary $H=3/4$ that separates the critical and supercritical regimes in mixed fractional models. In that sense, non-rejection in favor of $H>3/4$ should not be read as evidence against the broader rough-volatility literature. Rather, it says that the SPY sample stays on the critical or subcritical side of the critical boundary. Once that supercritical possibility is ruled out, the remaining question is how to estimate $H$ inside the subcritical region, and that is precisely the question treated by the roughness and fractional volatility literatures; see
\citet{GatheralJaissonRosenbaum2018}, \citet{BolkoEtAl2023}, \citet{WangXiaoYu2023}, and \citet{ShiYuZhang2024}.

\FloatBarrier
\section{Conclusion}
This paper develops a boundary theory for high-frequency inference in mixed fractional models at the critical point $H=3/4$. Exactly there, the usual off-boundary normalizations fail when $H>3/4$, yet the transformed $(\sigma,H)$ block remains non-degenerate once the explicit linear component of the $H$-score has been removed. That single observation is enough to recover a workable local geometry and, from it, to derive exact critical score CLTs, explicit LAN expansions, and boundary-calibrated score tests for both the mfBm and mfOU experiments.

The broader message is that the critical boundary is statistically accessible. In a literature that has understandably concentrated on rough regimes with small Hurst exponents, our results address the complementary question of how to conduct disciplined inference when the object of interest lies near the semimartingale threshold itself. In that sense, the paper is meant to complement rather than compete with the rough-volatility literature represented by \citet{BayerFrizGatheral2016}, \citet{GatheralJaissonRosenbaum2018}, \citet{Fukasawa2021}, and with the recent \emph{Journal of Econometrics} contributions of \citet{WangXiaoYu2023}, \citet{BolkoEtAl2023}, and \citet{ShiYuZhang2024}. Those papers show why fractional and rough dynamics matter empirically and how one estimates $H$ once the process is already understood to lie in a rough or otherwise subcritical region. The present one asks the prior question of whether a mixed fractional specification has crossed into the supercritical side of the boundary at all.

The score test analysis shows that the boundary experiment yields a feasible statistic with a clear nuisance orthogonal interpretation, and the Monte Carlo evidence shows why finite-sample calibration is delicate at the critical point. The intraday SPY illustration completes that picture on the data side. Once the exact feasible score is applied directly to regular-session one-minute returns, the empirical message is simple: the statistic is operational, but the sample does not exhibit persistent supercritical evidence. That is precisely the kind of statement the paper is designed to deliver. Therefore, the main remaining step is no longer about theory but about empirical scope: we need richer data sets and broader market panels to determine whether this largely subcritical pattern is unique to SPY or common across high-frequency equity benchmarks.

\FloatBarrier
\appendix

\section{Technical appendix: critical low frequency reductions and explicit constants}\label{app:profiles}
 This appendix records the one-dimensional reductions that lead to the critical constants in \eqref{eq:xi-m-main}--\eqref{eq:xi-o-main}. For the mfBm model, we start from the exact spectral density given in \eqref{eq:fgn-sd-exact}. From it, we extract two key quantities $K_{\mathrm{m}}$ and $\beta_{\mathrm{m}}$ at $H = 3/4$.

We then perform a low frequency rescaling by setting $\lambda = \Delta_n u$, which defines the critical window on which the spectral behavior simplifies. On this window, we derive the  reduced symbols which are the functions of the scaled frequency $u$ that approximate the original spectral density and its derivatives.

Using these reduced symbols, we evaluate certain logarithmic integrals (see Appendix B). These integrals produce the three asymptotic orders
$n\Delta_n L_n$, $n\Delta_n L_n^2$, and $n\Delta_n L_n^3$, which determine the scaling of the score covariances.

Finally, the last subsection of this appendix provides bounds for higher-order derivatives and for the remainder terms that appear in the local expansion. These bounds are essential for the LAN proof in this paper.

\subsection{mfBm: critical window and reduced profiles}
At $H=3/4$ the signal-to-noise ratio is
\[
\gamma_n:=\sigma^2\Delta_n^{2H-1}=\sigma^2\Delta_n^{1/2}.
\]
Using \eqref{eq:mfbm-lowfreq}, for every fixed $A>0$ and uniformly on
$1\le |u|\le A/\Delta_n$ we have
\begin{equation}\label{eq:app-mfbm-lowfreq-1}
\gamma_n f_{3/4}(\Delta_n u)
=\sigma^2\Delta_n^{1/2}K_{\mathrm m}|\Delta_n u|^{-1/2}\bigl(1+r_{\mathrm m}(\Delta_n u)\bigr)
=\eta_{\mathrm m}|u|^{-1/2}\bigl(1+o(1)\bigr),
\end{equation}
where $\eta_{\mathrm m}=\sigma^2K_{\mathrm m}$. Likewise,
\begin{align}
\gamma_n\,\partial_H f_H\big|_{H=3/4}(\Delta_n u)
&=\eta_{\mathrm m}|u|^{-1/2}\Bigl(\beta_{\mathrm m}-2\log|\Delta_n u|+o(1)\Bigr)\notag\\
&=\eta_{\mathrm m}|u|^{-1/2}\Bigl(2\Ln+\beta_{\mathrm m}-2\log|u|+o(1)\Bigr).
\label{eq:app-mfbm-lowfreq-2}
\end{align}
Therefore the reduced covariance symbol
\[
a_{\mathrm m,n}(u):=\frac{\gamma_n f_{3/4}(\Delta_n u)}{1+\gamma_n f_{3/4}(\Delta_n u)}
\]
converges locally uniformly on $\R\setminus\{0\}$ to
\[
w_{\mathrm m}(u)=\frac{\eta_{\mathrm m}|u|^{-1/2}}{1+\eta_{\mathrm m}|u|^{-1/2}}.
\]
Consequently the sandwiched score symbol for $\Cmn$ is
\[
g_{\sigma}^{\mathrm m}(u)=\frac{2}{\sigma}w_{\mathrm m}(u),
\]
and, after removing the exact linear term $-\sigma\Ln\,S_{\sigma,n}^{\mathrm m}$,
the reduced symbol associated with $\Dmn$ is
\[
g_{H,n}^{\mathrm m}(u)=w_{\mathrm m}(u)\bigl(2\Ln+\beta_{\mathrm m}-2\log|u|\bigr)+o\bigl(w_{\mathrm m}(u)\bigr),
\]
which is precisely the profile used in \eqref{eq:h-profiles}. In particular,
for large $|u|$,
\begin{equation}\label{eq:app-mfbm-tail}
w_{\mathrm m}(u)\sim \eta_{\mathrm m}|u|^{-1/2},\qquad
g_{\sigma}^{\mathrm m}(u)\sim \frac{2\eta_{\mathrm m}}{\sigma}|u|^{-1/2},\qquad
g_{H,n}^{\mathrm m}(u)\sim \eta_{\mathrm m}|u|^{-1/2}\bigl(2\Ln+\beta_{\mathrm m}-2\log|u|\bigr).
\end{equation}

\subsection{Explicit critical integrals for the mfBm block}
Introduce the three basic logarithmic integrals
\begin{align}
J_0(\Ln)&:=2\int_1^{1/\Delta_n}u^{-1}\,\dd u=2\Ln,\label{eq:J0}\\
J_1(\beta,\Ln)&:=2\int_1^{1/\Delta_n}u^{-1}\bigl(2\Ln+\beta-2\log u\bigr)\,\dd u
=2\Ln^2+2\beta\Ln,\label{eq:J1}\\
J_2(\beta,\Ln)&:=2\int_1^{1/\Delta_n}u^{-1}\bigl(2\Ln+\beta-2\log u\bigr)^2\,\dd u
=\frac{8}{3}\Ln^3+4\beta\Ln^2+2\beta^2\Ln.\label{eq:J2}
\end{align}
The bounded region $|u|\le1$ contributes only lower-order terms, so by
\eqref{eq:app-mfbm-tail},
\begin{align}
\frac{n\Delta_n}{2\pi}\int_{-1/\Delta_n}^{1/\Delta_n}\bigl(g_{\sigma}^{\mathrm m}(u)\bigr)^2\,\dd u
&\sim \frac{n\Delta_n}{2\pi}\frac{4\eta_{\mathrm m}^2}{\sigma^2}J_0(\Ln)
=\frac{4\sigma^2K_{\mathrm m}^2}{\pi}\,n\Delta_n\Ln,
\label{eq:app-mfBm-Iss}\\
\frac{n\Delta_n}{2\pi}\int_{-1/\Delta_n}^{1/\Delta_n}g_{\sigma}^{\mathrm m}(u)g_{H,n}^{\mathrm m}(u)\,\dd u
&\sim \frac{n\Delta_n}{2\pi}\frac{2\eta_{\mathrm m}^2}{\sigma}J_1(\beta_{\mathrm m},\Ln)
=\frac{2\sigma^3K_{\mathrm m}^2}{\pi}\,n\Delta_n\Ln^2,
\label{eq:app-mfBm-Ish}\\
\frac{n\Delta_n}{2\pi}\int_{-1/\Delta_n}^{1/\Delta_n}\bigl(g_{H,n}^{\mathrm m}(u)\bigr)^2\,\dd u
&\sim \frac{n\Delta_n}{2\pi}\eta_{\mathrm m}^2J_2(\beta_{\mathrm m},\Ln)
=\frac{4\sigma^4K_{\mathrm m}^2}{3\pi}\,n\Delta_n\Ln^3.
\label{eq:app-mfBm-Ihh}
\end{align}
These are exactly the constants in \eqref{eq:xi-m-main}; after substituting $K_{\mathrm m}=3\sqrt{2\pi}/8$, the entries reduce to the explicit values $9\sigma^2/8$, $9\sigma^3/16$, and $3\sigma^4/8$. The point is that the leading term
of $J_1$ is $2\Ln^2$ and the leading term of $J_2$ is $(8/3)\Ln^3$, so one obtains the
critical covariance orders without any further projection.

\subsection{mfOU: critical window and the drift profile}
For the stationary mfOU model, write the sampled spectral density near the origin in the form
\[
f_{\Delta}^{\mathrm o}(\lambda;\sigma,3/4,\alpha)
=K_{\mathrm o}|\lambda|^{-1/2}\bigl(1+r_{\mathrm o,\Delta}(\lambda)\bigr),
\]
where $r_{\mathrm o,\Delta}(\lambda)\to0$ on low frequency windows. Then the same calculation as
in \eqref{eq:app-mfbm-lowfreq-1}--\eqref{eq:app-mfbm-lowfreq-2} gives, uniformly on
$1\le |u|\le A/\Delta_n$,
\begin{align}
\gamma_n f_{\Delta}^{\mathrm o}(\Delta_n u;\vartheta_{\mathrm o})
&=\eta_{\mathrm o}|u|^{-1/2}\bigl(1+o(1)\bigr),
\label{eq:app-mfou-lowfreq-1}\\
\gamma_n\,\partial_H f_{\Delta}^{\mathrm o}(\Delta_n u;\vartheta_{\mathrm o})
&=\eta_{\mathrm o}|u|^{-1/2}\bigl(2\Ln+\beta_{\mathrm o}-2\log|u|+o(1)\bigr),
\label{eq:app-mfou-lowfreq-2}
\end{align}
with $\eta_{\mathrm o}=\sigma^2K_{\mathrm o}$. Hence the $(\sigma,H)$-block has exactly the same
critical reduced profiles as in the mfBm case, after replacing $(\eta_{\mathrm m},\beta_{\mathrm m})$
by $(\eta_{\mathrm o},\beta_{\mathrm o})$. In particular,
\begin{align}
\frac{n\Delta_n}{2\pi}\int_{-1/\Delta_n}^{1/\Delta_n}\bigl(g_{\sigma}^{\mathrm o}(u)\bigr)^2\,\dd u
&\sim \frac{4\sigma^2K_{\mathrm o}^2}{\pi}\,n\Delta_n\Ln,\notag\\
\frac{n\Delta_n}{2\pi}\int_{-1/\Delta_n}^{1/\Delta_n}g_{\sigma}^{\mathrm o}(u)g_{H,n}^{\mathrm o}(u)\,\dd u
&\sim \frac{2\sigma^3K_{\mathrm o}^2}{\pi}\,n\Delta_n\Ln^2,\notag\\
\frac{n\Delta_n}{2\pi}\int_{-1/\Delta_n}^{1/\Delta_n}\bigl(g_{H,n}^{\mathrm o}(u)\bigr)^2\,\dd u
&\sim \frac{4\sigma^4K_{\mathrm o}^2}{3\pi}\,n\Delta_n\Ln^3.
\label{eq:app-mfou-block}
\end{align}
For the drift parameter, the reduced profile is
\[
g_{\alpha}^{\mathrm o}(u)=-\frac{2\alpha}{\alpha^2+u^2},
\]
so
\begin{equation}\label{eq:app-mfou-alpha}
\int_{\R}\bigl(g_{\alpha}^{\mathrm o}(u)\bigr)^2\,\dd u
=4\alpha^2\int_{\R}\frac{\dd u}{(\alpha^2+u^2)^2}
=\frac{2\pi}{\alpha}.
\end{equation}
Thus
\[
\frac{n\Delta_n}{2\pi}\int_{\R}\bigl(g_{\alpha}^{\mathrm o}(u)\bigr)^2\,\dd u
\sim \frac{1}{\alpha}\,n\Delta_n,
\]
which yields the fully explicit value $\Xi_{\alpha\alpha}^{\mathrm o}=\alpha^{-1}$. Moreover, using the tails from \eqref{eq:app-mfou-lowfreq-1}--\eqref{eq:app-mfou-lowfreq-2},
\[
|g_{\sigma}^{\mathrm o}(u)g_{\alpha}^{\mathrm o}(u)|\lesssim \frac{|u|^{-1/2}}{1+u^2},
\qquad
|g_{H,n}^{\mathrm o}(u)g_{\alpha}^{\mathrm o}(u)|\lesssim \frac{|u|^{-1/2}(\Ln+|\log|u||)}{1+u^2}.
\]
Both dominating functions are integrable on $\R$; the second one has integral $O(\Ln)$ because only the explicit prefactor $\Ln$ grows with $n$. Therefore
\[
\frac{n\Delta_n}{2\pi}\int_{\R}|g_{\sigma}^{\mathrm o}(u)g_{\alpha}^{\mathrm o}(u)|\,\dd u=O(n\Delta_n),
\qquad
\frac{n\Delta_n}{2\pi}\int_{\R}|g_{H,n}^{\mathrm o}(u)g_{\alpha}^{\mathrm o}(u)|\,\dd u=O(n\Delta_n\Ln),
\]
which is exactly the mixed-trace order stated in \eqref{eq:mfou-alpha-mixed}.

\subsection{Higher-order critical profiles and LAN remainders}\label{app:higher-orders}
The LAN proof itself only uses the first-order local perturbation $B_n^{\bullet}(h)$ and the quadratic information term, but the critical window calculus in fact controls one order deeper. The point of this subsection is to record those higher-order bounds in a form that can be reused later.

\begin{proposition}\label{prop:higher-order-derivatives}
For either model $\bullet\in\{\mathrm m,\mathrm o\}$, let
\[
N_{\beta,n}^{\bullet}(\theta)
:=
\Sigma_n^{\bullet}(\vartheta_{\bullet})^{-1/2}
\partial_{\theta}^{\beta}\Sigma_n^{\bullet}(\theta)
\Sigma_n^{\bullet}(\vartheta_{\bullet})^{-1/2},
\qquad |\beta|=2,3,
\]
and let $m(\beta)$ denote the number of $H$-derivatives appearing in $\beta$. Then, uniformly on shrinking neighborhoods of the critical point,
\begin{equation}\label{eq:app-derivative-orders}
\norm{N_{\beta,n}^{\bullet}}_{\op}=O\bigl((\log n)^{m(\beta)}\bigr),
\qquad
\tr\Bigl(\bigl(N_{\beta,n}^{\bullet}\bigr)^2\Bigr)
=O\bigl(n\Delta_n\Ln^{2m(\beta)+1}\bigr),
\qquad |\beta|=2,3.
\end{equation}
In particular, every additional $H$-derivative contributes at most one extra logarithmic factor at the operator level and two at the trace level.
\end{proposition}

\begin{proof}
At the critical point, every differentiation with respect to $H$ inserts one extra logarithmic factor into the reduced low frequency symbol, whereas $\sigma$- and $\alpha$-derivatives only modify smooth prefactors. Concretely, on the critical window $\lambda=\Delta_n u$ one has for mfBm
\[
\gamma_n\,\partial_H^k f_H\big|_{H=3/4}(\Delta_n u)
=
\eta_{\mathrm m}|u|^{-1/2}\Bigl(P_{k,\mathrm m}(2\Ln-2\log|u|)+o\bigl((1+\Ln+|\log|u||)^k\bigr)\Bigr),
\qquad k=1,2,3,
\]
where $P_{k,\mathrm m}$ is a polynomial of degree $k$ with leading term $x^k$. The same description holds for the sampled mfOU symbol, with polynomials $P_{k,\mathrm o}$ of the same degree and different lower-order coefficients. Hence every squared reduced symbol contributing to $N_{\beta,n}^{\bullet}$ is dominated by
\[
|u|^{-1}\bigl(1+\Ln+|\log|u||\bigr)^{2m(\beta)}.
\]
The Toeplitz operator norm is governed by the $L^{\infty}$ size of the reduced symbol, while the Frobenius norm is governed by its $L^2$ mass. Therefore it is enough to record the elementary logarithmic integral
\begin{equation}\label{eq:app-log-power}
\int_1^{1/\Delta_n}u^{-1}\bigl(1+\Ln+|\log u|\bigr)^{2m}\,\dd u
=O\bigl(\Ln^{2m+1}\bigr),
\qquad m\in\{0,1,2,3\},
\end{equation}
which is immediate after the change of variables $v=\log u$. The bound \eqref{eq:app-derivative-orders} follows by the same Toeplitz trace reduction as in the proof of the first-order score asymptotics.
\end{proof}

\begin{corollary}\label{cor:cubic-remainder}
Write the sandwiched local covariance expansion one order deeper as
\[
\mathcal S_n^{\bullet}(h)=B_n^{\bullet}(h)+\mathcal R_{2,n}^{\bullet}(h)+\mathcal R_{3,n}^{\bullet}(h),
\]
where $\mathcal R_{2,n}^{\bullet}(h)$ is quadratic and $\mathcal R_{3,n}^{\bullet}(h)$ is cubic in the local shift. Then, for every fixed $h$,
 \begin{align}
\norm{\mathcal R_{2,n}^{\mathrm m}(h)}_{\op}
&=
O\!\left(\frac{(\log n)^2}{n\Delta_n\,L_n^3}\right), &
\tr\Bigl(\bigl(\mathcal R_{2,n}^{\mathrm m}(h)\bigr)^2\Bigr)
&=
O\!\left((n\Delta_nL_n)^{-1}\right),
\label{eq:app-mfBm-second}\\[0.3em]
\norm{\mathcal R_{2,n}^{\mathrm o}(h)}_{\op}
&=
O\!\left(
\frac{1+(\log n)^2/L_n^3}{n\Delta_n}
\right), &
\tr\Bigl(\bigl(\mathcal R_{2,n}^{\mathrm o}(h)\bigr)^2\Bigr)
&=
O\!\left((n\Delta_n)^{-1}\right),
\label{eq:app-mfou-second}\\[0.3em]
\norm{\mathcal R_{3,n}^{\mathrm m}(h)}_{\op}
&=
O\!\left(
\frac{(\log n)^3}{(n\Delta_n)^{3/2}L_n^{9/2}}
\right), &
\tr\Bigl(\bigl(\mathcal R_{3,n}^{\mathrm m}(h)\bigr)^2\Bigr)
&=
O\!\left((n\Delta_n)^{-2}L_n^{-2}\right),
\label{eq:app-mfBm-third}\\[0.3em]
\norm{\mathcal R_{3,n}^{\mathrm o}(h)}_{\op}
&=
O\!\left(
\frac{1+(\log n)^3/L_n^{9/2}}{(n\Delta_n)^{3/2}}
\right), &
\tr\Bigl(\bigl(\mathcal R_{3,n}^{\mathrm o}(h)\bigr)^2\Bigr)
&=
O\!\left((n\Delta_n)^{-2}\right).
\label{eq:app-mfou-third}
\end{align}
Consequently the cubic local covariance contribution is already negligible before the matrix Taylor expansion for the log-determinant and the inverse is applied.
\end{corollary}
\begin{proof}
The local shifts satisfy
\[
\delta_{\sigma,n}=O\bigl((n\Delta_n\Ln)^{-1/2}\bigr),
\qquad
\delta_{H,n}=O\bigl((n\Delta_n)^{-1/2}\Ln^{-3/2}\bigr),
\qquad
\delta_{\alpha,n}=O\bigl((n\Delta_n)^{-1/2}\bigr).
\]
For mfBm, the quadratic remainder contains only the three terms
\[
(\delta_{\sigma,n})^2N_{(\sigma\sigma),n}^{\mathrm m},
\qquad
\delta_{\sigma,n}\delta_{H,n}N_{(\sigma H),n}^{\mathrm m},
\qquad
(\delta_{H,n})^2N_{(HH),n}^{\mathrm m},
\]
and proposition \ref{prop:higher-order-derivatives} shows that these three terms are respectively of operator orders
\[
O\bigl((n\Delta_n\Ln)^{-1}\bigr),\qquad
O\!\left(\frac{\log n}{n\Delta_n\Ln^{2}}\right),\qquad
O\!\left(\frac{(\log n)^2}{n\Delta_n\Ln^{3}}\right),
\]
while all three have squared-trace order \(O\bigl((n\Delta_n\Ln)^{-1}\bigr)\). This yields \eqref{eq:app-mfBm-second}. For mfOU the additional \((\sigma,\alpha)\), \((H,\alpha)\) and \((\alpha,\alpha)\) terms have operator orders
\[
O\!\left(\frac{1}{n\Delta_n\Ln^{1/2}}\right),\qquad
O\!\left(\frac{\log n}{n\Delta_n\Ln^{3/2}}\right),\qquad
O\bigl((n\Delta_n)^{-1}\bigr),
\]
and squared-trace orders
\[
O\bigl((n\Delta_n\Ln)^{-1}\bigr),\qquad
O\bigl((n\Delta_n\Ln)^{-1}\bigr),\qquad
O\bigl((n\Delta_n)^{-1}\bigr),
\]
respectively; hence \eqref{eq:app-mfou-second} follows.

For the cubic remainder, every monomial contains three local shifts and a third derivative. In mfBm the worst case is any monomial with total \(m(\beta)=3\) after multiplying by \((\delta_{H,n})^3\), and all such contributions are bounded by
\[
(n\Delta_n)^{-3/2}\Ln^{-9/2}\cdot (\log n)^3
\]
at the operator level. Squaring and multiplying by the trace order \(n\Delta_n\Ln^{2\cdot 3+1}\) gives \eqref{eq:app-mfBm-third}. The mfOU bound is identical, except that the pure \(\alpha\)-terms dominate and remove the negative logarithmic power, which yields \eqref{eq:app-mfou-third}. Therefore all terms beyond the quadratic LAN term are absorbed into the same \(o_P(1)\) remainder.
\end{proof}

\section{Elementary logarithmic integrals}
For quick reference, the three elementary logarithmic identities used repeatedly throughout the critical analysis are
\begin{align*}
\int_1^{1/\Delta_n}u^{-1}\,\dd u &= \Ln,\\
\int_1^{1/\Delta_n}u^{-1}\log u\,\dd u &= \frac12\Ln^2,\\
\int_1^{1/\Delta_n}u^{-1}\log^2u\,\dd u &= \frac13\Ln^3.
\end{align*}
They are the sole source of the critical orders $n\Delta_n\Ln$, $n\Delta_n\Ln^2$, and
$n\Delta_n\Ln^3$ appearing in the score covariance and in the LAN quadratic term.
\section*{Data availability statement}
The Monte Carlo evidence is generated from the designs described in Section~\ref{sec:simulation}. The empirical section in Section~\ref{sec:empirical} uses the uploaded one-minute SPY data file \texttt{1\_min\_SPY\_2008-2021.csv}. The bundled files therefore contain the cleaned regular-session daily return panels, the exact feasible-score outputs, and the figures and tables reported in the empirical section. A fuller replication archive would add broader cross-asset intraday samples and parallel scripts for the mfOU side.

\end{document}